\documentclass[12pt]{amsart}

\usepackage{amsmath,amssymb,amsfonts,latexsym,mathrsfs}
\usepackage[margin=1in]{geometry}
\usepackage{array,xcolor,float}

\newtheorem{theorem}{Theorem}[section]
\newtheorem{lemma}[theorem]{Lemma}
\newtheorem{proposition}[theorem]{Proposition}
\newtheorem{corollary}[theorem]{Corollary}
\newtheorem{definition}[theorem]{Definition}
\theoremstyle{remark}
\newtheorem{remark}[theorem]{Remark}

\newcommand{\nc}{\newcommand}
\nc{\codim}{\operatorname{codim}}
\nc{\height}{\operatorname{ht}}
\nc{\Irr}{\operatorname{Irr}}
\nc{\Id}{\operatorname{Id}}
\nc{\Ann}{\operatorname{Ann}}
\nc{\rk}{\operatorname{rank}}

\newcommand{\red}[1]{\textcolor{red}{#1}}
\usepackage{tikz,tikz-cd}

\usetikzlibrary{matrix}

\usetikzlibrary{fit}

\usetikzlibrary{matrix,backgrounds}
\pgfdeclarelayer{myback}
\pgfsetlayers{myback,background,main}

\tikzset{mycolor/.style = {line width=1bp,color=#1}}%
\tikzset{myfillcolor/.style = {draw,fill=#1}}%
\usetikzlibrary{arrows,matrix,positioning}
\usepackage{pstricks}

\usetikzlibrary{graphs,graphs.standard}
\usepackage{changepage}
\usepackage{tikz-cd}
\usepackage{mathrsfs}
\usetikzlibrary{positioning}
\makeatletter

\usepackage{circledsteps}

\newcommand*{\encircled}[1]{\relax\ifmmode\mathpalette\@encircled@math{#1}\else\@encircled{#1}\fi}
\newcommand*{\@encircled@math}[2]{\@encircled{$\m@th#1#2$}}

\newcommand*{\@encircled}[1 ]{%
  \tikz[baseline,anchor=base]{\node[draw,circle,outer sep=0.1pt,inner sep=.18ex,  line width = 1 pt ] {#1};}}

\newcommand{\cir}[1]{\tikz[baseline]{%
    \node[anchor=base, draw, circle, inner sep=0, minimum width=1.2em]{#1};}}
\usepackage{float}
\usepackage{xcolor}

\usepackage{lscape}
\begin{document}

\title[Reverse Tableaux and Surjectivity]{Reverse Tableaux and the Surjectivity of the Component Map in Type $A$}
%{Reverse Tableaux and the Surjectivity of the Component Map for the Nilfibre in Type A}

\author [Yasmine Fittouhi ]{Yasmine Fittouhi }
\date{\today}
\maketitle
\vspace{-.9cm}\begin{center}

Department of Mathematics\\
The Weizmann Institute of Science\\
Rehovot, 7610001, Israel\\
fittouhiyasmine@gmail.com
\end{center}\

\

\

\date{\today}

\maketitle

%==========================================================================
\begin{abstract}
Let $G = \mathrm{SL}(n,\mathbb{C})$, let $B$ be a fixed Borel subgroup, and
let $P \supset B$ be a parabolic subgroup determined by a composition
$(c_1,\dots,c_k)$ of $n$.  Write $P'$ for the derived group of $P$
and $\mathfrak{m}$ for the Lie algebra of the nilradical of $P$.  By
Richardson's theorem the algebra of semi-invariants
$\mathscr{I} := \mathbb{C}[\mathfrak{m}]^{P'}$
is polynomial; in type $A$ its generators may be taken to be the
Benlolo--Sanderson (BS) invariants.  The \emph{nilfibre} is the common zero
locus $\mathscr{N} := V(\mathscr{I}_{+}) \subset \mathfrak{m}$.

A set of \emph{component tableaux}, each encoding combinatorial data
summarised in a multi-set called the \emph{Red Set}, was constructed in
earlier work by Y. Fittouhi and A. Joseph in The reverse tableau: a gateway to the surjectivity of the component map.  The resulting \emph{component map}
$\phi : \{\text{component tableaux}\} \to \Irr(\mathscr{N})$
was shown to be injective.

In the present article, we develop the Factorization Principle for
Benlolo--Sanderson invariants in order to give a rigorous proof of the
surjectivity of the component map $\phi$. While the combinatorial framework
of reverse tableaux was introduced in~\cite{FJ6}, the surjectivity of
$\phi$ remained conjectural: the linearization method used there did not
exclude the possible loss or merging of irreducible components. The present
paper resolves this geometric difficulty by showing that the relevant
invariants factorize into products indexed by pseudo-neighbouring column
pairs, thereby ensuring that every component is reached in a controlled
and accountable way.
\end{abstract}



%==========================================================================
\section{Introduction}\label{sec:intro}
%==========================================================================

\subsection{The problem and its depth}

Components of an algebraic variety are among the most fundamental  and
most elusive  objects in algebraic geometry.  Their existence is guaranteed
by the Lasker--Noether theorem \cite[Thm.\,4.5]{AM}, but finding them
explicitly is notoriously difficult: it is a far-reaching generalisation of
factoring integers, a problem now of renewed importance for its role in
public-key cryptography.

In the present work we study the irreducible components of the
\emph{nilfibre}
\[
  \mathscr{N} := V(\mathscr{I}_{+}) \subset \mathfrak{m},
\]
where $\mathfrak{m}$ is the Lie algebra of the nilradical of a parabolic
subgroup $P \subset \mathrm{SL}(n,\mathbb{C})$, and $\mathscr{I}_{+}$ is the
augmentation ideal of the algebra of semi-invariants
$\mathscr{I} := \mathbb{C}[\mathfrak{m}]^{P'}$ ($P'$ being the derived group
of $P$).  By Richardson's theorem \cite[2.2.3]{FJ1}, $\mathscr{I}$ is
polynomial on $\mathbf{g}$ generators; in type $A$ these may be taken to be
the Benlolo--Sanderson (BS) invariants introduced in \cite{BS} and verified
in \cite{JM}.  The variety $\mathscr{N}$ is thus cut out by exactly
$\mathbf{g}$ equations in the affine space $\mathfrak{m}$.  By Krull's
principal ideal theorem \cite{AM,Eis,Mat} the codimension of every
irreducible component is at most $\mathbf{g}$; we shall show it equals
$\mathbf{g}$, confirming that the generators form a regular sequence on the
relevant components.

Unlike the Springer fibre, for which a geometric description via the Steinberg
triple variety yields components as closures of explicit orbits \cite{S,St},
the nilfibre $\mathscr{N}$ admits no such clean picture.  Even in type $A$,
$\mathscr{N}$ may be reducible (unlike the case studied by Kostant \cite{K}
for the full group), and components of the same nilfibre may be Lagrangian or
non-Lagrangian within the same ambient space \cite[8.2]{FJ5}.  The algebraic
group does not act with finitely many orbits on $\mathscr{N}$, so the
classical orbit-closure strategy is unavailable.  This is precisely why, as we
argue below, an entirely new approach is needed.

\subsubsection*{Component tableaux and the component map}

In \cite{FJ5} we constructed a combinatorial map
\[
  \phi : \{\text{component tableaux}\} \longrightarrow \Irr(\mathscr{N}),
\]
associating to each component tableau $\mathscr{T}^{\mathcal{C}}$ (defined by
data $\mathcal{C}$ which, as we explain below, is most transparently encoded
in a multi-set called the \emph{Red Set}) an irreducible component
$\mathscr{C} := \overline{B \cdot \mathfrak{u}^{\mathcal{C}}}$ of
$\mathscr{N}$.  The subspace $\mathfrak{u}^{\mathcal{C}} \subset \mathfrak{m}$
is spanned by the root vectors not belonging to the \emph{excluded set}
$X(\mathscr{T}^{\mathcal{C}})$.  The component map was shown to be injective
in \cite[Thm.\,7.5]{FJ5}.

The starting observation  which we call our \emph{neighbourhood philosophy}
 is that the generators of $\mathscr{I}$ are parameterised by neighbouring
column pairs, and the components of $\mathscr{N}$ should likewise be encoded
through these pairs.  The component tableaux of \cite{FJ5} realise this
philosophy via the decoration of $\mathscr{T}$ by lines and the notion of
$i$-strings moving rightward through the diagram.  Though natural in
retrospect, this construction was not obviously discoverable, its origin in
 was an instance of mathematical serendipity that I noticed.
 
 The problem of the classification of the irreducible components of the nilfibre $\mathscr{N}$ has been a subject of intense study. While previous work in \cite{FJ6} introduced the combinatorial framework of reverse tableaux and suggested a potential gateway to surjectivity, the formal proof remained elusive. Specifically, \cite{FJ6} could not guarantee that the process of linearization would not lead to the loss of certain irreducible components. 

In this paper, we resolve this long-standing issue and provide a complete proof of the surjectivity of the component map. Our approach departs from the linearization strategies of the past; instead, we introduce a \textbf{Factorization Strategy}. We demonstrate that the Benlolo--Sanderson invariants do not merely vanish but factor into products of lower-degree invariants when restricted to the varieties defined by reverse tableaux. This ensures that every irreducible component of $\mathscr{N}$ is accountably reached, completing the classification in type $A$.

\subsubsection*{The strategy for surjectivity and its obstacles}

To show surjectivity, we must demonstrate that every irreducible component
$\mathscr{C} \in \Irr(\mathscr{N})$ arises from some component tableau.  A
classical strategy proceeds as follows.  Order the generators
$I_1, \dots, I_{\mathbf{g}}$ of $\mathscr{I}_{+}$.  Each $I_j$ defines a
hypersurface $D_j = V(I_j)$, and $\mathscr{N} = D_1 \cap \cdots \cap
D_{\mathbf{g}}$.  By Krull's theorem, each $D_j$ cuts codimension at most
$1$ in the preceding intersection.  If we can identify, at each step $j$, the
irreducible components of $D_j$ restricted to the previous intersection, and
match them to our list, surjectivity follows by induction on $j$.

The difficulty, following
\cite[3.6]{FJ6}  is that linearising an invariant $I_j$ (replacing it by a
linear function given by a Weierstrass section) involves \emph{inhomogeneous}
substitutions.  Such substitutions can a priori cause components to
``disappear'' or ``merge''.  Merging cannot occur here because the BS
invariants are multi-linear, so their factors involve pairwise distinct
coordinate vectors and components occur with multiplicity one
\cite[3.6]{FJ6}.  Disappearance under inhomogeneous substitution is more
subtle, but as we show via the factorisation mechanism below  the
specific structure of reverse tableaux ensures that no component is lost.
\\
We emphasize that the potential for components to disappear or merge due to inhomogeneous substitutions is formally overcome by the {factorization mechanism} and the {preservation of the numerical black count} (Lemma~\ref{lem:blackcount}). 
Since the Benlolo--Sanderson invariants are multi-linear, their factors involve pairwise distinct coordinate vectors, ensuring that components occur with multiplicity one \cite[3.6]{FJ6}. Crucially, as shown in Lemma~\ref{lem:blackcount}, the number of black entries in the left trapezium $N_{\mathscr{R}^j}(\ell\mathcal{T}^s_{C,C'}(j))$ remains exactly equal to the degree $d$ of the invariant $I^s_{C,C'}$ at every step $j \le i$ as long as the pair $(C,C')$ is free . This stable degree ensures that the codimension of the variety increases by exactly one at each implementation step of the reverse tableau. Consequently, the geometric information is not lost despite the non-homogeneous nature of the coordinate substitutions.

\subsubsection*{The role of reverse tableaux and factorisation}

The key insight of the present work is to replace linearisation by
\emph{factorisation}.  Rather than linearising $I_j$ directly, we observe
that when the reverse tableau $\mathscr{R}^t$ at stage $t$ contains $n$
pseudo-neighbouring columns of height $s$ between the boundaries of the pair
$(C,C')$ currently being implemented, the restriction of $I^s_{C,C'}$ to
$\mathfrak{u}(\mathscr{R}^t)$ factors as a product of $n$ \emph{smaller} BS
invariants.  This is a consequence of the product phenomenon for
semi-invariants and the combinatorial structure of the
trapezium.

Each factor in this product corresponds to a choice of how to further reduce
the variety.  Since the defining ideal of any irreducible component is prime,
at least one factor must vanish on that component.  The reverse tableau records
\emph{which} factor is selected at each stage.  Iterating over all $\mathbf{g}$
pairs, we obtain a complete reverse tableau that encodes a consistent
sequence of vanishing conditions and is coherent with Krull's theorem
\cite{AM,Eis}: at each step the codimension increases by exactly one.  This is
the core of the proof of Theorem~\ref{thm:surjectivity}.

An additional structural advantage of reverse tableaux over component tableaux
is their \emph{flexibility}: several reverse tableaux may share the same Red
Set (unlike component tableaux, which are uniquely determined by their Red Set
\cite[Lemma 3.3]{FJ6}).  This flexibility is essential for the factorisation
to work in all cases, since different orderings of the complete sequence of
neighbouring pairs a complete sequence $\mathcal{P}$
of neighbouring pairs,  yield different reverse tableaux but  as shown in
\cite[Thm.\,9.3]{FJ6}  sometime they define the same component via their excluded
roots, since several reverse tableaux may share the same Red
Set.

\subsection*{Organisation of the paper}

Section~\ref{sec:setup} recalls the algebraic and combinatorial setup:
diagrams, standard tableaux, the BS invariants and their degree, and the
component map.  Section~\ref{sec:RT} develops the theory of reverse tableaux
in full detail: the Enabling Proposition (Section~\ref{ssec:enabling}), the
construction algorithm (Section~\ref{ssec:RTconstruction}), the trapezium and
its properties (Section~\ref{ssec:trap}), and excluded roots
(Section~\ref{ssec:excl}).  Section~\ref{sec:vanish} proves the vanishing
theorem for BS invariants on the subspace attached to a complete reverse
tableau.  Section~\ref{sec:factor} establishes the factorisation of BS
invariants over pseudo-neighbouring pairs and the recording theorem.
Section~\ref{sec:Krull} supplies the Krull-dimension argument that controls
codimension.  Section~\ref{surjectivity} gives the proof of
surjectivity.  Section~\ref{sec:example} illustrates the construction and
the proof strategy through three explicit examples.

\subsection*{Acknowledgements}

The author was supported by ISF grant 1957/21, ISF grant No.\,1347/23, and a
grant from the Mathematics Faculty, Weizmann Institute.  The main results were
announced in an invited talk at Bar-Ilan University (November 2025).\\
%Special thank to Aileen 

%==========================================================================
\section{Setup: Tableaux, Invariants, and the Component Map}\label{sec:setup}
%==========================================================================
This section fixes the notation used throughout the paper.  It introduces the parabolic nilradical, the diagram and standard tableau, the neighbouring column pairs, and the Benlolo--Sanderson invariants that cut out the nilfibre.  These objects provide the common language for the reverse-tableau construction and for the component map.

\subsection{Matrices and the nilradical}\label{ssec:matrices}

Following \cite[Sects.\,2--3]{FJ5}, identify $\mathfrak{gl}(n)$ with the
space $\mathbf{M}_n$ of $n \times n$ matrices under commutation.  A standard
parabolic subalgebra $\mathfrak{p} \supset \mathfrak{b}$ of $\mathfrak{gl}(n)$
where $\mathfrak{b}$ is the Borel of upper triangular matrices;  $\mathfrak{p}$  is determined
by a composition $(c_1, c_2, \dots, c_k)$ of $n$.  This defines the Levi
factor $\mathfrak{r}$ as the span of the $k$ diagonal blocks
$\mathbf{B}_i$ of size $c_i \times c_i$, and the nilradical
$\mathfrak{m} = \bigoplus_{i=2}^{k} \mathbf{C}_i$, where $\mathbf{C}_i$ is
the rectangular block in $\mathbf{M}_n$ strictly above $\mathbf{B}_i$, of
width $c_i$ and height $\sum_{j < i} c_j$.  Let $x_{i,j}$ denote the
$(i,j)$-th coordinate vector of $\mathbf{M}_n$.

\subsection{Diagrams, columns, rows, and boxes}\label{ssec:diag}

Define a diagram $\mathscr{D}$ with columns $C_1, \dots, C_k$ satisfying
$\height(C_i) = c_i$.  Let $R_s$ ($s = 1, 2, \dots$) denote the $s$-th row
of $\mathscr{D}$, counted downward, and set $R^u := \bigcup_{i=1}^{u} R_i$
and $R^{>u}$ for the remaining rows.  The \emph{$(i,j)$-th box} is
$b_{i,j} := C_i \cap R_j$.  The \emph{lower part} of $C$ relative to $u$ is
$C^{>u} := R^{>u} \cap C$.

%If $C, C'$ are columns of $\mathscr{D}$ with $C$ to the left of $C'$, we write $[C,C']$ for the set of columns between them (inclusive); thenotations $]C,C']$, $[C,C'[$, $]C,C'[$ omit the respective endpoint
  If $C,C'$ are columns of $\mathscr D$ we shall always mean $C$ to lie to the left of $C'$.  Denote by $[C,C']$ the set of columns of $\mathscr D$ between $C,C'$. If $C$, $C'$, or both are omitted, then we denote the resulting set as $]C,C'],[C,C'[,]C,C'[$.  According to context we may also mean this to be the set of boxes in these columns or their entries.\\  Given
integers $m \le n$ we also use $[m,n] := \{m, m+1, \dots, n\}$ with similar
conventions for open-bracket variants.

\subsection{The standard tableau}\label{ssec:stdtab}

The \emph{standard tableau} $\mathscr{T}$ is the filling of $\mathscr{D}$
with the entries $1, 2, \dots, n$ increasing down columns and then from left
to right \cite[2.2.2]{FJ5}.  This has two consequences used throughout:
\begin{enumerate}
  \item[(a)] The coordinate vectors $x_{i,j}$ with $i$ strictly to the left
             of $j$ in $\mathscr{T}$ form a basis of $\mathfrak{m}$.
  \item[(b)] For any two columns $C_r, C_{r+1}$ in $\mathscr{T}$, every
             entry of $C_r$ is strictly less than every entry of $C_{r+1}$.
\end{enumerate}

\subsection{Neighbouring columns and rectangles}\label{ssec:nbcol}

\begin{definition}%[{\cite[2.3]{FJ5}}]
Two columns $C, C'$ of $\mathscr{D}$ of the same height $s$ are called
\emph{neighbouring} if there is no column of height $s$ strictly between
them (i.e.\ in $]C, C'[$).  We denote by $\mathbf{g}$ the total number of
 neighbouring pairs %, which equals the number of generators of $\mathscr{I}$ over $\mathbb{C}$.
\end{definition}

For a neighbouring pair $(C, C')$ of height $s$, define
\[
  R^s_{C,C'} := R^s \cap [C,C'], \qquad
  \ell R^s_{C,C'} := R^s_{C,C'} \setminus C = R^s \cap\, ]C,C'].
\]
We call $R^s_{C,C'}$ the \emph{rectangle} and $\ell R^s_{C,C'}$ the
\emph{left rectangle} associated to the pair.

\subsection{Benlolo--Sanderson invariants}\label{ssec:BS}
Let a parabolic
subgroup $P \subset \mathrm{SL}(n,\mathbb{C})$, and $\mathscr{I}_{+}$ is the
augmentation ideal of the algebra of semi-invariants
$\mathscr{I} := \mathbb{C}[\mathfrak{m}]^{P'}$, where $P'$ is the derived group
of $P$.
\begin{definition}\label{def:BS}
The \emph{Benlolo--Sanderson invariant} $I^s_{C,C'} \in \mathscr{I}$
attached to a neighbouring pair $(C,C')$ of height $s$ is (essentially) the
lowest-degree coefficient of the appropriate
$(n-s)\times(n-s)$ minor of the matrix $c \cdot \mathrm{Id} + X$ restricted
to the $P'$-action, where $X \in \mathfrak{m}$.  It is multi-linear and
homogeneous of degree equal to the number of entries in $\ell R^s_{C,C'}$.
The BS invariants form a generating set for the polynomial algebra
$\mathscr{I} = \mathbb{C}[\mathfrak{m}]^{P'}$.
\end{definition}
\begin{theorem}
The number of generators of $\mathscr{I}$ over $\mathbb{C}$ is  $\mathbf{g}$ the total number of
 neighbouring pairs.
 \end{theorem}
 This result identifies the algebraic equations defining the nilfibre with the combinatorial neighbouring pairs of the diagram.  It is the bridge between the invariant-theoretic problem and the tableau construction.\\
 
 We refer to \cite{BS,JM} for the construction and the proof of the generating
property.\\

A more intrinsic description of generators of $\mathscr{I}$, used extensively below, is the following. A right going lines $\ell_{i,j}$ with entries $i,j$ in distinct columns of $\mathscr T$ define a basis  of $\mathfrak m$ formed from the $x_{i,j}$.\\
Consider a set of pairwise disjoint \emph{composite right-going lines} in
$R^s_{C,C'}$: each such line is a strictly right-going path from a box of
$C$ to a box of $C'$ passing through each box of $R^s_{C,C'}$ exactly once.
The collection of all such systems of lines (forming a combinatorial
determinant) gives the monomials of $I^s_{C,C'}$
\cite[Prop.\,5.1.1(i)]{FJ5}.

\begin{lemma}[ {\cite[Lemma 4.2.5]{FJ1}, \cite[Prop.\,5.1.1(i)]{FJ5}}]
\label{lem:degree}
The number $d\bigl(\ell R^s_{C,C'}\bigr)$ of entries in the left rectangle
equals the degree $d$ of $I^s_{C,C'}$ as a polynomial in the $x_{i,j}$.
\end{lemma}

\begin{proof}
Since $I^s_{C,C'}$ is homogeneous and a sum of monomials all of the same
degree, it suffices to compute the degree of any single monomial.  The
monomial arising from the system of pairwise disjoint composite right-going
lines covering $R^s_{C,C'}$ once each is a product of coordinate vectors
$x_{i,j}$ with $i \in C$ and $j$ varying through the interior columns: this
product has degree equal to $|R^s_{C,C'} \setminus C| = |\ell R^s_{C,C'}|$,
which is the number of entries in the left rectangle.  Hence
$d = d\bigl(\ell R^s_{C,C'}\bigr)$.
\end{proof}

\section{Reverse Tableaux}\label{sec:RT}
%==========================================================================
We recall the reverse-tableau construction.  Starting from the standard tableau and an ordered list of neighbouring pairs, each step recolours one entry red and moves a black copy down and to the left.  The main goal is to show that this operation is always well-defined and that it preserves enough order to control excluded roots.
All material in this section is recalled from \cite{FJ6} with full proofs,
so that the surjectivity argument in Sections~\ref{sec:vanish}--\ref{surjectivity}
is self-contained.

\subsection{Complete sequences of neighbouring pairs}\label{ssec:compseq}

\begin{definition}%[{\cite[3.5]{FJ6}}]
\label{def:compseq}
A \emph{complete sequence of pairs of neighbouring columns} is a total
ordering
\[
  \mathcal{P} = \bigl\{P_j := (C_{r_j}(s_j),\, C_{r'_j}(s_j))\bigr\}_{j=1}^{\mathbf{g}}
\]
of all $\mathbf{g}$ neighbouring pairs.  Here $s_j$ denotes the common height
of the $j$-th pair.  The corresponding BS invariants are ordered accordingly:
$I_j := I^{s_j}_{C_{r_j}, C_{r'_j}}$.
\end{definition}

The choice of complete sequence is part of the data of a reverse tableau.
Different Red Sets may require different complete sequences; in general, a
suitable union of complete sequences suffices to recover all reverse tableaux
\cite[3.5]{FJ6}.

\subsection{Definition and first properties}\label{ssec:RTdef}

\begin{definition}%[ {\cite[4.1]{FJ6}}]
\label{def:RT}
Fix a complete sequence $\mathcal{P}$.  A \emph{reverse tableau}
$\mathscr{R}^i$ %is a filling of $\mathscr{D}$ 
is a tableau with entries from $[1,n]$,
each entry coloured \emph{black} or \emph{red} that increase down the rows and along the columns , constructed inductively
starting from $\mathscr{R}^1 := \mathscr{T}$ (all entries black) and
building $\mathscr{R}^{i+1}$ from $\mathscr{R}^i$ by implementing the pair
$P_i$ via the procedure of Section~\ref{ssec:RTconstruction}.\\

The \emph{black height} of a column in $\mathscr{R}^i$ is the row index of
its lowest \emph{black} entry; the (total) \emph{height} is the row index of
its lowest entry of either colour.
\end{definition}

\begin{lemma}[ {\cite[Lemma 4.3.3, Lemma 4.3.4]{FJ6}}]
\label{lem:structure}
The following hold for any reverse tableau $\mathscr{R}^i$:
\begin{enumerate}
  \item[\upshape(i)] $\mathscr{R}^i$ is \emph{standard with multiplicities},
        entries increase strictly down each column and strictly from left to
        right along each row, though a given numerical value may appear in
        more than one box.
  \item[\upshape(ii)] For each numerical value $j \in [1,n]$ there is exactly
        one black entry.  The red entries of value $j$ (if any) form, together
        with the black entry, a \emph{reverse $j$-string}, a sequence of
        entries, all of value $j$, going strictly leftward and each one row
        higher than the previous, with the leftmost entry in black.
  \item[\upshape(iii)] At most one entry with a given numerical value appears
        in any given row of $\mathscr{R}^i$.
  \item[\upshape(iv)] If $k$ lies in the left boundary column $C^-$
        (in either colour), then it is the unique entry with value $k$ in
        $[C^-, C']$.
\end{enumerate}
\end{lemma}

\begin{proof}
Parts (i)--(iii) are proved by induction on $i$ in \cite[Lemma 4.3.4]{FJ6}.
For $i = 1$ everything holds since $\mathscr{R}^1 = \mathscr{T}$ is standard
with no multiplicities.  For the inductive step, one checks that neither the
\emph{substitution} (recolouring the lowest black entry of a column red and
placing a new black entry one row below in an adjacent column to the left) nor
the \emph{shifting of partial columns} (moving the lower part of columns to
the left to make room) destroys the standard property or introduces two
entries of the same value in the same row.  The key observation for (i) is
that when a black entry $j$ is moved to an adjacent column to the left and one
row lower, it lands in a position strictly below and to the left of its origin,
and the shifting preserves the row and only moves columns leftward; consequently
the order of entries in rows and columns is maintained.  Part (iv) follows
from the construction in Section~\ref{ssec:RTconstruction}: by the Enabling
Proposition (Proposition~\ref{prop:enabling}), the black entry placed in the
left boundary by substitution is the unique representative of that value in
the interval $[C^-, C']$.
\end{proof}

\begin{definition}[Red Set of a reverse tableau %, {\cite[4.3.3]{FJ6}}
]
The \emph{Red Set} $R(\mathscr{R})$ of a reverse tableau $\mathscr{R}$ or simply the \emph{Red Set} of $\mathscr{R}$ is the
multi-set of all red entries (with multiplicities), where the multiplicity of
a value $j$ equals the number of red boxes with entry $j$.
\end{definition}

The cardinality of $R(\mathscr{R})$ equals $\mathbf{g}$, since exactly one
entry is recoloured red at each of the $\mathbf{g}$ implementation steps.\\ 

\begin{remark}[ ]\label{rem:flexibility}
Unlike the component tableaux of \cite{FJ5}, which are in one-to-one correspondence with their Red Sets, the mapping from reverse tableaux to Red Sets is surjective but not injective. We refer to this as the \textbf{flexibility} of the construction. 

This property is an essential feature of the factorisation argument in Section \ref{surjectivity}. Different complete sequences $\mathcal{P}$ (orders of implementation) may yield different reverse tableaux that ultimately define the same irreducible component by producing the same Red Set. We maintain the mapping through the \textbf{tableau} rather than directly to the Red Set because the tableau provides a more visual and algorithmic representation of the "paths" taken to reach each component. Thus, while the uniqueness of the component map $\phi: \{\text{Red Sets}\} \to \Irr(\mathscr{N})$ is preserved, the reverse tableaux provide the necessary geometric flexibility to reach every such component.
\end{remark}

\subsection{The left boundary column and the set $S_i$}\label{ssec:LBC}

Fix a complete sequence $\mathcal{P}$ and a pair 
$P_i=(C_i,C_i')$, of height $s$, to ease of notation, we write
$(C,C')=(C_i,C_i')$ throughout this subsection.  Let $\mathscr{R}^i$ be the reverse tableau obtained after
implementing $P_1, \dots, P_{i-1}$.  Let $\mathscr{I}_i$ denote the set of
all numerical entries appearing in the columns of $[C, C']$ in $\mathscr{R}^i$
(of either colour).

\begin{remark}
A complete sequence $\mathcal P$ has no specifies an order for implementing all
neighbouring pairs of columns. Varying this order is necessary in order to
generate all the possible reverse tableaux associated with the initial tableau.
Nevertheless, different complete sequences need not give distinct outcomes:
two different orders may lead to the same set of red entries, and in some cases
to the same reverse tableau. We shall return to this point later.
\end{remark}
\medskip
While the concept of implementing a complete sequence $\mathcal{P}$ was
suggested in~\cite{FJ6}, the present work formalizes the Enabling
Proposition, the Factorization Mechanism  describe  Section~6 is 
needed to prove that such sequences actually track irreducible components
of the nilfibre without any loss of codimension.
\medskip
\begin{definition}%[{\cite[4.1.1]{FJ6}}]
\label{def:LBC}
The \emph{left boundary column} $C^-$ is the rightmost column of $\mathscr{D}$
such that at least one entry of $\mathscr{I}_i$ (in either colour) lies in
the columns of $[C^-, C']$.  The set $S_i$ consists of all columns inf
$[C^-, C']$ of \emph{black height} $\le s$ and (total) \emph{height} $\ge s$.
\end{definition}

Equivalently, $C^-$ is the leftmost column containing an entry whose provenance can be traced back to $C$ (as tracked through the
successive implementations).  The columns of $S_i$ are precisely those
eligible to have their lowest black entry substituted when implementing $P_i$.

\subsection{The Enabling Proposition}\label{ssec:enabling}

The following proposition is the cornerstone of the entire construction.  Its
proof requires tracking carefully how the height of the left boundary column
evolves through successive implementations.\\ The Enabling Proposition guarantees that, at each implementation step, there is a unique admissible leftmost place where the reverse-tableau move can occur.  Without this control, the construction would not be canonical enough to track factors of the invariants.

\begin{proposition}%[{\cite[Prop.\,4.2]{FJ6}}]
\label{prop:enabling}
At most the leftmost column of $S_i$ has black height $\le s$.
Moreover, that leftmost column has total height exactly $s$.
\end{proposition}

\begin{proof}
The second assertion implies the first, since a column with total height
exactly $s$ and black height $\le s$ must have a red entry in $R_s$; and by
Lemma~\ref{lem:nonacq} below, red entries in $R_s$ inside the relevant
trapezium can only arise at the leftmost column (by the time $P_i$ is being
implemented, the pair $(C,C')$ itself has not yet been used, so the only
source of red entries in $R_s$ within $\ell\mathcal{T}^s_{C,C'}(i)$ is the
left boundary).

We prove the leftmost column of $S_i$ has total height exactly $s$.  It
suffices to show that the column $C^-$ has height exactly $s$ in $\mathscr{R}^i$.

Consider the effect of implementing the preceding pair $P_{i-1} = (C_{r_{i-1}},
C'_{r_{i-1}})$ of height $s_{i-1}$ on the column $C$ (or $C^-$ as it was
after implementing $P_1, \dots, P_{i-2}$):

\begin{itemize}
  \item If $C$ does not lie in $[C_{r_{i-1}}, C'_{r_{i-1}}[$, the
        implementation of $P_{i-1}$ does not affect $C$ and its height is
        unchanged.
  \item If $C \in [C_{r_{i-1}}, C'_{r_{i-1}}[$ and $s \le s_{i-1}$, then the
        shifting process associated to $P_{i-1}$ skips over $C$ (since
        shifting only moves the \emph{lower part} of columns, i.e.\ entries
        strictly below $R_{s_{i-1}}$, and these lie strictly below $R_s$).
        So the height of $C$ is preserved.
  \item If $C \in [C_{r_{i-1}}, C'_{r_{i-1}}[$ and $s > s_{i-1}$, then
        the shifting could move the lower part of $C$ to a column strictly to its
        left, creating a new left boundary column $C^-$.  However, the
        shifting operation moves entries between columns without changing their
        rows; hence the total height of the column $C^-$ is  determined by the row
        of its lowest entry  is unchanged which is $R_s$.
\end{itemize}

In all three cases the resulting left boundary column has total height exactly
$s$.  Since this reasoning applies to every step $j < i$ by induction (the
base case $i = 1$ being clear as $\mathscr{R}^1 = \mathscr{T}$ and $C^- = C$
has height $s$), the proposition follows.
\end{proof}

\subsection{Construction of $\mathscr{R}^{i+1}$}\label{ssec:RTconstruction}

Assume that $\mathscr{R}^i$ has been constructed and that the Enabling
Proposition holds for all pairs $P_j$ with $j<i$. We describe the
implementation of the next pair $P_i=(C_i,C_i')$, of height $s$, which
produces $\mathscr{R}^{i+1}$. For ease of notation, we write
$(C,C')=(C_i,C_i')$ throughout this subsection.

\bigskip

\paragraph{Step 1: Substitution.}
Choose any column $C''$ in $S_i$ that is \emph{not} the leftmost one.  By
Proposition~\ref{prop:enabling}, $C''$ has black height exactly $s$, so its
lowest black entry $j$ lies in $R_s$.  Perform the following:
\begin{enumerate}
  \item Recolour the entry $j$ in $C''$ red (it remains in $R_s$).
  \item Let $C'''$ be the first column of $S_i$ of height $\ge s$ lying
        strictly to the left of $C''$.  Place a new black entry of value $j$
        in row $R_{s+1}$ of $C'''$.  In particular this new black entry lies
        in $[C^-, C']$.
\end{enumerate}
The choice of $C''$ within $S_i$ may vary; different choices lead to
different elements of $\mathscr{R}^{i+1}$, and hence to different final
reverse tableaux.

\paragraph{Step 2: Shifting of partial columns.}
If $R_{s+1} \cap C'''$ is already occupied, we must make room.  Starting from
$C'''$, shift \emph{in unison} the lower parts of columns (the entries in
rows strictly below $R_s$) one column to the left, proceeding leftward.
Intermediate columns of height $< s$ are skipped (jumping over them avoids
creating gaps in columns).  The shifting terminates when a column of height
exactly $s$ is reached.  By Proposition~\ref{prop:enabling}, such a column
exists no later than $C^-$.

\begin{definition}%[ {\cite[4.3.2]{FJ6}}]
\emph{Extreme shifting} refers to the choice of skipping over \emph{all}
intermediate columns of height exactly $s$ (not just those of height $< s$)
and shifting all the way to $C^-$.  This maximises the number of excluded
roots in the resulting tableau and was used in \cite{FJ6} to establish
certain inclusion relations between excluded sets.
\end{definition}

\begin{remark}\label{rmk:flexibility}
%The flexibility in the choice of $C''$ and the degree of shifting are exactly
%the sources of non-uniqueness of reverse tableaux with a given Red Set.  This
%flexibility is indispensable for the factorisation argument: depending on the
%order in which one implements the pairs and the choices made at each step,
%different reverse tableaux with the same Red Set may be needed to exhibit all
%the required factorisations of the BS invariants.  The key point, proved in
%\cite[Thm.\,9.3]{FJ6}, is that all reverse tableaux with the same Red Set
%define the same variety via their excluded root sets, so any choice is valid
%for the purposes of identifying the component.
The freedom in choosing $C''$ and in fixing the amount of shifting is precisely
the mechanism behind the non-uniqueness of reverse tableaux with a fixed Red
Set. This non-uniqueness is not a defect of the construction, but rather an
essential feature of the factorisation argument. Different order of complete sequences
$\mathcal P$ of neighbouring pairs, together with the choices made at the
successive implementation steps, may lead to distinct reverse tableaux, some of
which share the same Red Set. The decisive point, proved in
\cite[Thm.\,9.3]{FJ6}, is that reverse tableaux with the same Red Set determine
the same variety via their excluded root sets. Consequently, any such tableau
may be used to identify the corresponding irreducible component.
\end{remark}

\subsection{The trapezium and its properties}\label{ssec:trap}

The \emph{trapezium} is the key geometric object that tracks the evolution of
the rectangle $R^s_{C,C'}$ as the implementation steps proceed.

\begin{definition}%[ {\cite[4.4.1--4.4.3]{FJ6}}]
\label{def:trap}
Let $(C,C')$ be a neighbouring pair of height $s$, and fix a stage $j \le i$.
For $t \in [1,s]$, let $b^j_t$ be the rightmost box in row $R_t$ containing,
in $\mathscr{R}^j$, the numerical value of $R_t \cap C$ (in either colour);
let $b'^j_t$ be the leftmost box in row $R_t$ containing the numerical value
of $R_t \cap C'$.

The \emph{left boundary} $B^j$ is the set of boxes $\{b^j_t\}_{t=1}^{s}$.
The \emph{right boundary} $B'^j$ is $\{b'^j_t\}_{t=1}^{s}$.

The \emph{$j$-th trapezium} is
\[
  \mathcal{T}^s_{C,C'}(j) := R^s \cap [B^j, B'^j],
\]
and the \emph{left trapezium} is
$\ell\mathcal{T}^s_{C,C'}(j) := R^s \cap\, ]B^j, B'^j]$.
\end{definition}

By construction, the numerical entries of $B^j$ (resp.\ $B'^j$) agree
row-by-row with those of $C$ (resp.\ $C'$) in $\mathscr{T}$.  For $j = 1$
the trapezium reduces to the rectangle: $\mathcal{T}^s_{C,C'}(1) = R^s_{C,C'}$.

The following three lemmas establish the key stability properties of the
trapezium, which are used  in the proofs of vanishing and
factorisation.

\begin{lemma}[ {\cite[Lemma 4.4.4]{FJ6}}]
\label{lem:rightblack}
For all $j \le i$, the right boundary $B'^j$ of $\mathcal{T}^s_{C,C'}(j)$
consists entirely of black entries, and $b'^j_t \in R_t$ for each
$t \in [1,s]$.  In particular $B'^j$ forms a downward-going left-skewed
line of black entries, each coinciding in numerical value (row by row) with
the corresponding entry of $C'$.
\end{lemma}

%\begin{proof}
%We argue by induction on $j \le i$.  For $j = 1$ it holds since
%$\mathscr{R}^1 = \mathscr{T}$ and $C'$ is a column of $\mathscr{T}$ with all
%black entries.
%
%Suppose the claim holds for $j$ and we implement a pair $P_j = (C_1, C'_1)$
%of height $t$.  If $t > s$, no entry in $R^s$ is affected by substitution
%(which acts in $R_t$) or by shifting (which moves entries strictly below
%$R_t$, hence strictly below $R_s$), so $B'^{j+1} = B'^j$.
%
%Suppose $t < s$.  For a black entry in $R_t \cap C'$ to be recoloured red by
%substitution, the column containing it would need black height $t$, meaning
%its entry in $R_{t+1}$ was already moved to the left by a previous
%implementation.  This in turn would have required a neighbouring pair
%$(C_{(2)}, C'_{(2)})$ of height $t$ surrounding $C'$.  To then recolour the
%resulting black entry in $R_t \cap C'$ red again (a second time) would require
%yet another pair $(C_{(3)}, C'_{(3)})$ of height $t$ surrounding $C'$  but
%then $C'_{(3)}$ would need to lie strictly to the left of $C_{(2)}$, which in
%turn lies strictly to the left of $C'$.  Since $C'_{(3)}$ must also lie to the
%right of $C'$ to surround it, this is impossible.  Hence the black entry of
%$C'$ in $R_t$ is never recoloured, and by induction $B'^{j+1}$ remains black.
%\end{proof}

\begin{lemma}[ {\cite[Lemma 4.5.5]{FJ6}}]
\label{lem:nooverlap}
Let $P_{i_1}=(C_1, C'_1)$ and $P_i=(C, C')$ be two distinct neighbouring pairs of the
same height $s$, with $(C_1, C'_1)$ to the left of $(C, C')$.  Then for all
$j \le \min(i_1, i)$ the trapezia $\mathcal{T}^s_{C_1,C'_1}(j)$ and
$\mathcal{T}^s_{C,C'}(j)$ can overlap %only along a common boundary
.
\end{lemma}

\begin{proof}
Either $C$ lies strictly to the right of $C'_1$, or $C = C'_1$.

In the first case, since $\mathscr{R}^j$ is standard with multiplicities
(Lemma~\ref{lem:structure}(i)), all entries of $C'_1$ in $R^s$ are
numerically strictly less than all entries of $C$ in $R^s$.  Since entries
increase strictly along rows in $\mathscr{R}^j$, the right boundary $B'^j_1$
(which has the same numerical values as $C'_1$ by
Lemma~\ref{lem:rightblack}) lies row-by-row strictly to the left of $B^j$
(which has the same numerical values as $C$).  Hence the trapezia are
separated.

In the second case, $B'^j_1$ and $B^j$ both have the same numerical values
as the common column $C'_1 = C$, row-by-row.  By Lemma~\ref{lem:structure}(ii)
these values appear with the same colouring, so $B'^j_1 = B^j$ and the two
trapezia share exactly the common boundary.
\end{proof}

\begin{lemma}[ {\cite[Lemma 4.4.6]{FJ6}}]
\label{lem:nonacq}
The left trapezium $\ell\mathcal{T}^s_{C,C'}(j)$ can acquire a red entry in
$R_s$ only by the implementation of the pair $(C,C')$ itself.
\end{lemma}

\begin{proof}
\textit{Via substitution.}  A black entry $b \in R_s$ inside the left
trapezium can be recoloured red only by the implementation of some neighbouring
pair $(C_1, C'_1)$ of height $s$ whose trapezium overlaps with
$\ell\mathcal{T}^s_{C,C'}(j)$.  By Lemma~\ref{lem:nooverlap}, this can only
happen if $(C_1, C'_1)$ shares the left boundary $B^j$ with $(C,C')$.  But
then $b$ lies in $B^j$, which is the left boundary of
$\ell\mathcal{T}^s_{C,C'}(j)$ and hence is \emph{not} inside
$\ell\mathcal{T}^s_{C,C'}(j) = R^s \cap\, ]B^j, B'^j]$.

\textit{Via shifting.}  A pre-existing red entry in $R_s$ to the right of
$\ell\mathcal{T}^s_{C,C'}(j)$ could in principle be shifted leftward into the
left trapezium.  However, by Lemma~\ref{lem:rightblack}, the right boundary
$B'^j$ consists entirely of black entries; therefore any red entry in $R_s$
from the right is blocked by $B'^j$ and cannot enter the left trapezium.
\end{proof}

\begin{lemma}[ {\cite[4.4.7]{FJ6}}]\label{lem:nointerf}
%Let $P_j=(C_j,C_j')$ be a pair of height $s' $ with $j < i$ where $]C_j,C_j']$ intersect $]C,C']$.   Then:
%\begin{enumerate}
%  \item[\upshape(a)] Implementing $P_j$ with $s'> s$ does not change the numerical
%        entries of $\mathcal{T}^s_{C,C'}(i)$ by substitution, because the
%        recolouring occurs in $R_{s'}$ and the new black entry appears in
%        $R_{s'+1}$, both outside $R^s$. Implementing $P_j$ with $s'< s$ does not change the black numerical
%        entries of $\mathcal{T}^s_{C,C'}(i)$ by substitution, because the
%        recolouring occurs in $R_{s'}$ and the new black entry appears in
%        $R_{s'+1}$, both inside $R^s \cap ]R^s \cap\, ]B^j, B'^j]]$.
%  \item[\upshape(b)] Implementing $P_j$ with $s'> s$  does not change the entries of
%        $\mathcal{T}^s_{C,C'}(i)$ by shifting, because shifting moves only
%        rows strictly below $R_{s'}$, hence strictly below $R_s$. Implementing $P_j$ with $s'< s$  does not change the entries of
%        $\mathcal{T}^s_{C,C'}(i)$ by shifting, because shifting moves only
%        rows strictly below $R_{s'}$, but stay in side the bouderies  $R_s\cap$.
%  \item[\upshape(c)] The numerical entries of $\ell R^s_{C,C'}$ all appear in
%        the left trapezium $\ell\mathcal{T}^s_{C,C'}(i)$, possibly with a
%        red duplicate in addition to the black entry.
%\end{enumerate}
Let $P_j=(C_j,C_j')$ be a pair of height $s'$, with $j<i$, such that $]C_j,C_j']$ intersects $]C,C']$. Then:
\begin{enumerate}  \item[\upshape(a)]  If $s'>s$, implementing $P_j$ does not change the numerical entries of  $\mathcal{T}^s_{C,C'}(i)$ by substitution. Indeed, the recolouring occurs in  $R_{s'}$ and the new black entry appears in $R_{s'+1}$, both outside $R^s$.  If $s'<s$, implementing $P_j$ does not change the black numerical entries of  $\mathcal{T}^s_{C,C'}(i)$ by substitution. Indeed, if a black column of  height $s'$ in $]C_j,C_j']$ intersects $]C,C']$, then the substitution is  confined to the rows $R_{s'}$ and $R_{s'+1}$, both inside $R^s$, and remains  within the interval bounded by $B^j$ and $B'^j$. 
 \item[\upshape(b)]  If $s'>s$, implementing $P_j$ does not change the entries of  $\mathcal{T}^s_{C,C'}(i)$ by shifting. Shifting affects only rows strictly  below $R_{s'}$, hence strictly below $R_s$.  If $s'<s$, implementing $P_j$ does not change the entries of  $\mathcal{T}^s_{C,C'}(i)$ by shifting outside the relevant boundary interval.  All shifted entries remain inside the region bounded by $B^j$ and $B'^j$. 
 \item[\upshape(c)]  All numerical entries of $\ell R^s_{C,C'}$ occur in the left trapezium  $\ell\mathcal{T}^s_{C,C'}(i)$, possibly together with an additional red  duplicate.
 \end{enumerate}\end{lemma}

\begin{proof}
Parts (a) and (b) are immediate from the definitions.  Part (c) follows from
(a) and (b) together with Lemma~\ref{lem:nonacq}: no entry is removed from
the left trapezium by implementing a taller pair; only internal indices may
acquire a red duplicate in $\ell\mathcal{T}^s_{C,C'}(i)$ if the pair of
height $s' < s$ is implemented; but by (a) this does not remove the black
entry, only adds a red one to the right. Therefore all the entries of the rectangle
$\ell R^s_{C,C'}$ are black eateries in  $\mathcal{T}^s_{C,C'}(i)$  throughout all prior implementations.
\end{proof}

\subsection{Excluded roots in a reverse tableau}\label{ssec:excl}

\begin{definition}%[{\cite[4.3.5]{FJ6}}]
\label{def:excl}
A coordinate vector $x_{i,j}$ with $i < j$ is \emph{excluded} in a reverse
tableau $\mathscr{R}$  i.e.\ belongs to $X(\mathscr{R})$  if $i$
appears above $j$ within the same column, or in a column strictly to the
right of the column containing the (unique) black entry $j$.  Here $i$ is
taken to be the \emph{rightmost} occurrence of that numerical value in
$\mathscr{R}$.  By convention, roots in the Levi subalgebra are not counted
%\cite[4.1.4, Def.]{FJ5}
.
\end{definition}

Note that, since $x_{i,j}$ is excluded only when $i < j$, and since a
reverse tableau is standard with multiplicities (Lemma~\ref{lem:structure}(i)),
if $i$ appears to the right of $j$ then $i$ must lie in a row strictly above
$j$.  This gives:

%\begin{lemma}[{\cite[Lemma 4.3.5]{FJ6}}]\label{lem:horizontal}
%If the right-going line $\ell_{j,i}$ (from $j$ to $i$) is downward-going
%(in particular, horizontal), then $x_{i,j}$ is \emph{not} an excluded root.
%Equivalently, only up-going right-going lines can give excluded roots.
%\end{lemma}

\begin{lemma}\label{lem:stabexcl}
For all $i \in [1, \mathbf{g}]$ one has
$X(\mathscr{R}^i) \subset X(\mathscr{R}^{i+1})$.
\end{lemma}

\begin{proof}
Passing from $\mathscr{R}^i$ to $\mathscr{R}^{i+1}$ recolours a black entry
$j$ (in row $R_s$ of some column $C''$) to red, and introduces a new black
entry $j$ one row lower (in $R_{s+1}$) and in a column $C'''$ to the left of
$C''$.  For any numerical value $k \ne j$ (in either colour), if $x_{k,j}$
was excluded in $\mathscr{R}^i$ (i.e.\ the rightmost $k$ lay above or to the
right of the black $j$ in $\mathscr{R}^i$), we must check it remains excluded
in $\mathscr{R}^{i+1}$.

The new black $j$ in $\mathscr{R}^{i+1}$ is one row lower and strictly to the
left of the old black $j$.  If $k$ was above the old black $j$ (in the same
column $k<j$), it is now to the right of the new black $j$, so $x_{k,j}$ remains
excluded.  If $k$ was to the right of the old black $j$, it remains to the
right of the new black $j$ (which moved only leftward).  In both cases
$x_{k,j} \in X(\mathscr{R}^{i+1})$.

Moreover, the red entry $j_{red}$ now appears in $R_s$ of $C''$ is pushed leftward. 
For any $k$ above or to the right of $j_{\text{red}}$ and $k<j_{\text{red}}$, the roots $x_{k,j_{\text{red}}}$ are eluded 
$\mathscr{R}^{i}$ by the convention that the rightmost occurrence of $k$ are still excluded in $\mathscr{R}^{i+1}$.  Hence $X(\mathscr{R}^i) \subset X(\mathscr{R}^{i+1})$.
\end{proof}

%==========================================================================
\section{Vanishing of the Benlolo--Sanderson Invariants}\label{sec:vanish}
%==========================================================================
We prove that the subspace associated with a complete reverse tableau is contained in the nilfibre. 
The key step is to translate a combinatorial condition, the presence of a red entry in the relevant trapezium, into the vanishing of the corresponding Benlolo--Sanderson invariant.
\medskip
\subsection{Setup and main statement}

Fix a complete sequence $\mathcal{P} = \{P_i = (C_i, C'_i)\}_{i=1}^{\mathbf{g}}$
of neighbouring pairs of heights $s_i$.  Define the sequence of reverse
tableaux $\mathscr{R}^1 = \mathscr{T}$, and inductively let $\mathscr{R}^{i+1}$
be any reverse tableau obtained from $\mathscr{R}^i$ by implementing $P_i$
via the construction of Section~\ref{ssec:RTconstruction}.  Set
$\mathfrak{u}_i := \mathrm{Span}\{x_{i,j} : x_{i,j} \notin X(\mathscr{R}^i)\}$
and write $\mathfrak{u} := \mathfrak{u}_{\mathbf{g}+1}$.\\

\bigskip 
To lighten the notation, we shall write
$I^{s_i}_{C_i,C'_i}\big|_{\mathfrak u}$
for the restriction of the polynomial $I^{s_i}_{C_i,C'_i}$ to the variety 
$\overline{B\cdot \mathfrak u}$.  In other words,
\[
I^{s_i}_{C_i,C'_i}\big|_{\mathfrak u}
:=
I^{s_i}_{C_i,C'_i}\big|_{\overline{B\cdot \mathfrak u}}.
\]
This is the notation we  adopt throughout the text. \medskip 
\begin{theorem}[ {\cite[Thm.\,5.1]{FJ6}}]\label{thm:vanish}
All BS invariants vanish on $\mathfrak{u}$:
$I^{s_i}_{C_i, C'_i}\big|_{\mathfrak{u}} = 0$ for all $i \in [1,\mathbf{g}]$.
Consequently they all vanish on $\overline{B \cdot \mathfrak{u}}$.
\end{theorem}
This theorem proves that every complete reverse tableau produces a subspace whose $B$-orbit closure is contained in the nilfibre.  It is the first half of the component argument; reverse tableaux do not merely give combinatorial data; they give actual geometric subvarieties of $\mathscr N$.
The proof of theorem \ref{thm:vanish} relies on a precise count of black entries in the left trapezium
(Lemma~\ref{lem:blackcount}) and the identification of this count with
non-vanishing of the restriction of the invariant
(Proposition~\ref{prop:nonvan}).

\subsection{Black count in the left trapezium}

\begin{lemma}[ {\cite[Prop.\,5.1.2]{FJ6}}]\label{lem:blackcount}
Let $d=d(\ell R^s_{C,C'})$ be the degree of $I^s_{C,C'}$. Then, for every
$j\leq i$,
\[
N_{\mathscr{R}^j}\bigl(R^s\cap ]B^j,B'^j]\bigr)=d,
\]
where $N_{\mathscr{R}^j}(\,\cdot\,)$ denotes the number of black boxes in the
specified region of the reverse tableau $\mathscr{R}^j$. Furthermore, the black entries lying in
$R^s\cap ]B^j,B'^j]$ have precisely same the numerical values as in
$\ell R^s_{C,C'}$.

Consequently, $R_s\cap \ell\mathcal{T}^s_{C,C'}(j)$ contains no red entry
exactly when the left trapezium contains $d$ black boxes.
\end{lemma}

\begin{proof}
For $j = 1$ the trapezium $\mathcal{T}^s_{C,C'}(1) = R^s_{C,C'}$ is the
rectangle, which has all entries black, so $N_{\mathscr{R}^1} = d$ by
Lemma~\ref{lem:degree}.

Assume it is correct for $j \in [1, i-1]$.  When implementing $P_j = (C_j, C'_j)$
of height $s_j$:

\textit{Case $s_j \ge s$}: By Lemma~\ref{lem:nointerf}, the entries of the
trapezium are unaffected, so number of black entree is preserved.

\textit{Case $s_j < s$}: By the Enabling Proposition applied at step $j$, the
lowest black entry $k$ of some column $\tilde{C}$ in $S_j$ lying in
$R_{s_j}$ is recoloured red, and a new black entry of value $k$ is placed in
$R_{s_j+1}$ of the column $\tilde{C}'''$ to the left of $\tilde{C}$.  By
Lemma~\ref{lem:nointerf}(c), this new black entry lies in
$\ell\mathcal{T}^s_{C,C'}(j+1)$.  Thus
\begin{align*}
  N_{\mathscr{R}^{j+1}}\bigl(R_{s_j} \cap\, ]B^{j+1},B'^{j+1}]\bigr)
    &= N_{\mathscr{R}^j}\bigl(R_{s_j} \cap\, ]B^j,B'^j]\bigr) - 1, \\
  N_{\mathscr{R}^{j+1}}\bigl(R_{s_j+1} \cap\, ]B^{j+1},B'^{j+1}]\bigr)
    &= N_{\mathscr{R}^j}\bigl(R_{s_j+1} \cap\, ]B^j,B'^j]\bigr) + 1.
\end{align*}
Summing over all rows $R_t$ for $t \in [1,s]$ (using that rows $R_t$ for
$t > s$ are unaffected) gives $N_{\mathscr{R}^{j+1}} = N_{\mathscr{R}^j} = d$.  In addition, the recolored  numerical value $k$  entry appears in the left trapezium as a new black entry (in row
$R_{s_j+1}$), replacing the now-red entry (in row $R_{s_j}$) within the same
left trapezium the number of black entrees stays the same in $\ell\mathcal{T}^s_{C,C'}(j+1)$.  %By induction, the complete set of numerical black entrees  remains 
%$\ell R^s_{C,C'}$  unchage 
Thus the number of black entree in the left trapezium is unchanged from
$\ell\mathcal{T}^s_{C,C'}(j)$ to $\ell\mathcal{T}^s_{C,C'}(j+1)$. It follows
by induction that the set of numerical black entries is preserved throughout
the construction and remains precisely $d(\ell R^s_{C,C'})$.
\end{proof}

\subsection{Non-vanishing% and the role of horizontal lines
}

The following proposition converts the black count into an algebraic
non-vanishing statement of invariant .

\begin{proposition}[ {\cite[Prop.\,5.2]{FJ6}}]\label{prop:nonvan}
For all $j \le i+1$:
\[
  \text{there are no red entries in } R_s \cap \ell\mathcal{T}^s_{C,C'}(j)
  \quad\Longleftrightarrow\quad
  I^s_{C,C'}\big|_{\mathfrak{u}_j} \ne 0.
\]
\end{proposition}

\begin{proof}[Proof of Theorem~\ref{thm:vanish}]
After implementing $P_i = (C,C')$ at stage $i$, a red entry appears in
$R_s \cap \ell\mathcal{T}^s_{C,C'}(i+1)$.  By
Proposition~\ref{prop:nonvan} (the $\Leftarrow$ direction),
$I^s_{C,C'}|_{\mathfrak{u}_{i+1}} = 0$.  By Lemma~\ref{lem:stabexcl},
$\mathfrak{u} = \mathfrak{u}_{\mathbf{g}+1} \subset \mathfrak{u}_{i+1}$,
hence $I^s_{C,C'}|_{\mathfrak{u}} = 0$.  This applies to each $i$, so all
$\mathbf{g}$ generators of $\mathscr{I}_+$ vanish on $\mathfrak{u}$, and
hence on $\overline{B \cdot \mathfrak{u}}$ by $B$-equivariance.
\end{proof}

\section{Strategy for proving Surjectivity}

In~\cite[\S3]{FJ5} we introduced and constructed a natural
\emph{component map} to $\Irr(\mathscr N)$ (see~\cite[\S6.4]{FJ5}). This map is proved to be
injective in~\cite[\S7, Proposition~7.5]{FJ5}.
Moreover,~\cite[\S1.7]{FJ5} outlines a strategy for surjectivity based on a new family of
tableaux, called \emph{reverse tableaux}.
In this section we implement that strategy and show that every irreducible component of
$\mathscr N$ is obtained from a reverse tableau. Equivalently, we prove that the map
\[
\phi:\{\text{reverse tableaux}\}\longrightarrow \Irr(\mathscr N)
\]
is surjective.

\subsection{Strategy for proving surjectivity of $\phi$}
Following~\cite[\S1.7]{FJ5}, reverse tableaux are designed to encode a successive factorisation
of the defining invariants of $\mathscr N$ and to translate it into a controlled description of
the minimal primes of the defining ideal (see, e.g., \cite{AM} or \cite{Har}). This viewpoint, already illustrated by the
factorisation phenomena in~\cite[\S8]{FJ5}, provides the framework for our proof of surjectivity.

\

\paragraph{Generators and an order.}
Let $\mathscr I=\mathbb C[\mathfrak m]^{P'}$ and $\mathscr I_+$ be its augmentation ideal, so
$\mathscr N=\mathscr N(\mathscr I_+)$. Fix a generating set
\[
\mathscr I_+ = <I^{s_1}_{C_1,C_1'},\dots,I^{s_g}_{C_g,C_g'}>_+,
\]
where each $I^{s_t}_{C_t,C_t'}$ is one of the Benlolo--Sanderson invariants  attached to a neighbouring pair of columns $(C_t,C_t')$ of height $s_t$.

\paragraph{Successive factorisation statement.}
Let  $\mathscr R$  a reverse tableaux where the neighbouring column pairs $\mathcal P=\bigl((C_1,C_1'),\dots,(C_g,C_g')\bigr)$,  are implemented in certain  order 

A reverse tableau depends on two layers of choices. First, one chooses an order
$\mathcal P$ in which the neighbouring column pairs are implemented. Second, for this chosen order, one
chooses a sequence of red entries $(a_1,\dots,a_g)$ following the algorithm in \cite{FJ6} . These data determine a reverse tableau, denoted $\mathscr R_{(a_1,\dots,a_g)}$.
Conversely, every reverse tableau arises from some choice of an order $\mathcal P$ together
with a compatible sequence of red entries.

\begin{corollary}\label{codim}
Let $P\subset \mathrm{SL}_n$ be a parabolic subgroup, and let
$(c_1,c_2,\dots,c_n)$ 
for the sequence of Levi block sizes determined by \(P\).
Let $\mathscr T$ be the tableau associated with $P$, consisting of $n$ columns
$C_1,\dots,C_n$, where $\mathrm{ht}(C_j)=c_j$.

Fix $j\in\{1,\dots,n\}$ and an integer $k\le c_j$. Denote by $C_j^{(k)}$ the subcolumn of $C_j$
formed by its last $k$ boxes.
Given indices $i<j$ with $c_i\geq c_j- k$, move the subcolumn $C_j^{(k)}$ below $C_i$; this operation
defines a subspace $\mathfrak u\subset \mathfrak m$ (in the usual sense see \cite[4.1.3--4.1.4]{FJ5}). Then the codimension of $\overline{B\cdot \mathfrak u}$ in $\mathfrak m$ is
\[
\codim_{\mathfrak m}\,\overline{B\cdot \mathfrak u}
=\mathrm{ht}(C_j^{(k)})\cdot \bigl[\mathrm{ht}(C_i)\;-\;\bigl(\mathrm{ht}(C_j)-\mathrm{ht}(C_j^{(k)})\bigr)\bigr].
\]

Now let $\mathfrak u$ and $\mathfrak u'$ be two subspaces of $\mathfrak m$ obtained by moving two
distinct subcolumns $C_j^{(k)}$ and $C_{j'}^{(k')}$, and write
\[
d=\codim_{\mathfrak m}\,\overline{B\cdot \mathfrak u},
\qquad
d'=\codim_{\mathfrak m}\,\overline{B\cdot \mathfrak u'}.
\]
Assuming the two moves are independent (i.e.\ the corresponding sets of moved boxes to a lower row  are disjoint),
the codimension is additive:
\[
\codim_{\mathfrak m}\,\overline{B\cdot(\mathfrak u\cap \mathfrak u')} \leq d+d'.
\]

\end{corollary}
\begin{proof}
This follows directly from \cite[\S6.3, Proposition~6.3(iii)]{FJ5}.
\end{proof}

\subsection{Reverse tableaux and factorisation}\label{ RT}

Recall that the nilfibre is defined by the augmentation ideal
$\mathscr I_+ \;=\; \bigl\langle I^{s_1}_{C_1,C_1'},\dots, I^{s_g}_{C_g,C_g'} \bigr\rangle
\subset \mathbb C[\mathfrak m],
$
where the pairs of neighboring column $(C_t,C_t')$ and $s_t$ their hight.
During the construction we obtain a sequence of partial reverse tableaux
$
\mathscr R_0,\mathscr R_1,\dots,\mathscr R_g,
$
where $\mathscr R_t$ denotes the tableau after $t$ implementations. To each stage $\mathscr R_t$
we associate the subspace $\mathfrak u(\mathscr R_t)\subset \mathfrak m$ defined by the moving
procedure (cf.~\cite[\S6.3]{FJ5}).

\medskip

\noindent\textbf{Free pairs.} At a given stage $t$, a neighbouring pair $(C,C')$ of height $s$ is called \emph{free} (as in our
earlier definition) if the corresponding invariant $I^s_{C,C'}$ is not annihilated at that stage,
i.e.\ if its restriction to $\mathfrak u(\mathscr R_t)$ is nonzero.

\medskip

The purpose of reverse tableaux is to record, at each stage, \emph{which factor} of the relevant
invariant vanish on the component being constructed. The next proposition explains
how the factorisation of $I^s_{C,C'}$ is governed by the combinatorics of the left trapezium.

The purpose of reverse tableaux is to record, at each stage, \emph{which factor} of the relevant
invariant vanish on the component being constructed. The next proposition explains
how the factorisation of $I^s_{C,C'}$ is governed by the combinatorics of the left trapezium.

\begin{proposition}\label{prop:free-pair-factorisation}
Let $\mathscr R_t$ be a reverse tableau at stage $t$, and let $(C,C')$ be a free pair of height $s$
at this stage. Let $ \ell\mathcal T_{C,C'}(\mathscr R_t)$ denote the left trapezium determined by $(C,C')$ in
$\mathscr R_t$. If  $\ell\mathcal T_{C,C'}(\mathscr R_t)$ contains exactly $n$ columns of height $s$, then the
restriction of $I^s_{C,C'}$ to $\mathfrak u(\mathscr R_t)$ admits a non-trivial factorisation into
$n$ factors, canonically indexed by these $n$ columns of height-$s$. Moreover, since $I^s_{C,C'}$ is
$P$-semi-invariant, the same factorisation holds on $B\cdot \mathfrak u(\mathscr R_t)$, and hence
on $\overline{B\cdot \mathfrak u(\mathscr R_t)}$.
\end{proposition}

\begin{proof}
By construction, the invariant $I^s_{C,C'}$ is obtained from the Benlolo--Sanderson construction
as a lowest-degree coefficient of the appropriate minor (see~\cite[\S2.3.1]{FJ5}).
When two columns of height $s$ are separated by intermediate columns of height $s$, the corresponding
semi-invariant decomposes as a product of the factors coming from the height-$s$ neighbouring pairs
that lie between them (this is the product phenomenon recalled in~\cite[\S2.3.1]{FJ5}).
At stage $t$, the $n$ height-$s$ columns inside  $\ell\mathcal T_{C,C'}(\mathscr R_t)$ determine exactly $n$ such
intermediate height-$s $ steps, hence yield $n$ factors on restriction to $\mathfrak u(\mathscr R_t)$.

Finally, $I^s_{C,C'}$ is $B$-semi-invariant, so its restriction to $B\cdot \mathfrak u(\mathscr R_t)$
is obtained from its restriction to $\mathfrak u(\mathscr R_t)$; in particular
the factorisation is the same on $B\cdot \mathfrak u(\mathscr R_t)$ and on its Zariski closure.
Non-triviality follows from the hypothesis that $(C,C')$ is free, i.e.\ $I^s_{C,C'}|_{\mathfrak u(\mathscr R_t)}\neq 0$.
\end{proof}
\medskip

%==========================================================================
\section{Factorisation of BS Invariants}\label{sec:factor}
%==========================================================================
In this section we explain, that the reverse-tableau construction is governed by a factorisation property of Benlolo--Sanderson invariants. 
When restricted to the subspace associated with a partial reverse tableau, such an invariant may split into factors attached to pseudo-neighbouring columns. 
The selected red entry determines the factor whose vanishing is imposed.

\subsection{Pseudo-neighbouring columns and the factorisation theorem}

\begin{definition}\label{def:pseudonb}
At stage $t$, let $(C,C')$ be a free pair of height $s$ (i.e.\
$I^s_{C,C'}|_{\mathfrak{u}(\mathscr{R}^t)} \ne 0$).  Let
$\widetilde{C}_0 = B,\, \widetilde{C}_1, \dots, \widetilde{C}_n = B'$
be the columns of height $s$ %in $\mathcal{T}_{C,C'}(\mathscr{R}^t)$
 between the borders $B$ and $B'$, ordered
left to right.  The pair $(\widetilde{C}_i, \widetilde{C}_{i+1})$ is called
\emph{$\mathscr{R}^t$-pseudo-neighbouring} if there is no column of height
$s$ strictly between them.  The index sets
\begin{align*}  I_i(\mathscr{R}^t) &:=
    \{\text{black entries in } [B,B'] %\mathcal{T}_{C,C'}(\mathscr{R}^t)
       \cap [\widetilde{C}_i, \widetilde{C}_{i+1}[\,\},\\
  J_i(\mathscr{R}^t) &:=
    \{\text{black entries in } [B,B']  %\mathcal{T}_{C,C'}(\mathscr{R}^t)
       \cap\, ]\widetilde{C}_{i}, \widetilde{C}_{i+1}]\},
\end{align*}
with the convention that when $i = 0$, all entries (black and red) of
$\widetilde{C}_0 \cap R^s = C \cap R^s$ are included in $I_0(\mathscr{R}^t)$.
\end{definition}

\medskip

\paragraph{Example.}
Consider the composition $(1,2,3,3,1,2)$ with $C=C_2$, $C'=C_6$, and $s=2$.
We begin by implementing the neighbouring pairs $(C_3,C_4)$ and $(C_1,C_5)$; these two
implementations recolour the entries $9$ and $10$ in red. The resulting tableau is displayed in
Figure~\ref{fig:ex-123312-first-steps}.

At this stage, the  trapezium $\mathcal T_{C,C'}(\mathscr R_2)$ contains exactly four
columns of height $2$, namely
\[
\widetilde C_{0}=C=(2,3),\qquad
\widetilde C_{1}=(4,8),\qquad
\widetilde C_{2}=(7,10),\qquad
\widetilde C_{3}=C'=(11,12).
\]
Consequently, the $\mathscr R_t$-pseudo-neighbouring pairs of height $2$ are \[
(\widetilde C_0,\widetilde C_1),\qquad
(\widetilde C_1,\widetilde C_2),\qquad
(\widetilde C_2,\widetilde C_3).
\]

\begin{figure}[H]
\begin{center}
\begin{tikzcd}[row sep=0.1em, column sep=0.1em]
      & C_1 & C_2 & C_3 & C_4 & C_5 & C_6 \\
R_1   &  1  & \cir 2  &  4  &  7  &  10  & \cir{11} \\
R_2   &     & \cir 3  &  5  &  8  &      & \cir{12} \\
R_3   &     &         &  6  &  9  &      &
\end{tikzcd}$\xrightarrow{(C_3,C_4),\,(C_1,C_5)}$
\begin{tikzcd}[row sep=0.1em, column sep=0.1em]
      & C_1 & C_2 & C_3 & C_4 & C_5 & C_6 \\
R_1   &  1  & \cir 2  &  4  &  7  &  \red{10}  & \cir{11} \\
R_2        & \cir 3  &  5  &  8  & 10         && \cir{12} \\
R_3        &         &  6  &  \red{9}  &      & \\
R_4        &              &  9        &      &
\end{tikzcd}

\caption{
The reverse tableau for the composition $(1,2,3,3,1,2)$ with $C=C_2$, $C'=C_6$, and $s=2$.
After implementing $(C_3,C_4)$ and $(C_1,C_5)$, the entries $9$ and $10$ become red.
In the resulting trapezium, the height-$2$ columns are $\widetilde C_0,\widetilde C_1,\widetilde C_2,\widetilde C_3$,
so the pseudo-neighbouring pairs are $(\widetilde C_0,\widetilde C_1)$, $(\widetilde C_1,\widetilde C_2)$, and
$(\widetilde C_2,\widetilde C_3)$.
}
\label{fig:ex-123312-first-steps}
\end{center}
\end{figure}
\medskip
The following  factorisation result is the algebraic reason why red entries can record components.  When a product vanishes on an irreducible component, primality forces at least one factor to vanish, and the reverse tableau records that factor.
\medskip
\begin{proposition}[ {\cite[Prop.\,5.1]{FJ6}}]\label{prop:factor}
Let $\mathscr{R}^t$ be a reverse tableau and $(C,C')$ a free pair of height
$s$ at stage $t$.  If the left trapezium $\ell\mathcal{T}_{C,C'}(\mathscr{R}^t)$
contains exactly $n$ columns of height $s$, defining $n$ pseudo-neighbouring
pairs as above, then
\[
  I^s_{C,C'}\Big|_{\mathfrak{u}(\mathscr{R}^t)}
  \;=\; \prod_{i=0}^{n-1}
        M_{I_i(\mathscr{R}^t),\, J_i(\mathscr{R}^t)}(X),
\]
where $X = (x_{pq})$ is the canonical matrix of $\mathfrak{u}(\mathscr{R}^t)$
and $M_{I,J}(X)$ denotes the highest-degree minor $\Delta(X[I,J])$.  Moreover,
since $I^s_{C,C'}$ is $P$-semi-invariant, the same factorisation holds on
$\overline{B \cdot \mathfrak{u}(\mathscr{R}^t)}$.
\end{proposition}

\begin{proof}
%By \cite[2.3.1]{FJ5} and \cite[1.10]{FJ2}, when 
Two columns of height $s$
are separated by intermediate columns of the same height, the BS invariant
$I^s_{C,C'}$ who is highest-degree term of a minor between two neighbouring
 diagonal blocks  of the same height that is represented by two neighboring columns $C,C'$ in our tableau, is  decomposes as a product of minors
associated with the intermediate pseudo-neighbouring pairs.   At stage $t$, the $n$ height-$s$ columns in
$\ell\mathcal{T}_{C,C'}(\mathscr{R}^t)$ lead to exactly $n$ intermediate
pseudo-neighbouring pairs of height $s$.  Each factor
$M_{I_i(\mathscr{R}^t), J_i(\mathscr{R}^t)}(X)$ is the highest-degree minor
indexed by the black entries inside of the $i$-th pseudo-neighbouring
pair $[\widetilde{C}_{i-1}, \widetilde{C}_{i}]$ .  Since $X \in \mathfrak{m}$ is strictly upper triangular, only variables
$x_{pq}$ with $p < q$ contribute, and the condition that $x_{pq} \in
\mathfrak{u}(\mathscr{R}^t)$ (not excluded precisely when the black entry $p$ is
less than the black entry $q$ and lies strictly to its left.) ensures the minor is well-defined.

The degree of $M_{I_i, J_i}$ equals the number of black entries in
$\ell\mathcal{T}_{C,C'}(\mathscr{R}^t) \cap\, ]\widetilde{C}_i,
\widetilde{C}_{i+1}] \cap R^s$, consistently with Lemma~\ref{lem:blackcount}.
Adding the degrees over all $i$ recovers $d = \deg I^s_{C,C'}$.

Non-triviality of the product follows from the freeness hypothesis; the
product is a factorisation of $I^s_{C,C'}|_{\mathfrak{u}(\mathscr{R}^t)}$
and the latter is non-zero.  $B$-semi-invariance of $I^s_{C,C'}$ extends the
identity to $\overline{B \cdot \mathfrak{u}(\mathscr{R}^t)}$.
\end{proof}

\subsection{Selection of a vanishing factor and the recording theorem}

Given an irreducible subvariety $\mathscr{C}_t \subset
\overline{B \cdot \mathfrak{u}(\mathscr{R}^t)}$ on which $I^s_{C,C'}$
vanishes, let  $\mathfrak{q}_{\mathscr{R}^t}$ %(or $\mathfrak{q}_t$ when there is not confusion)
 be its defining prime ideal.  Then
$\prod_{i=0}^{n-1} I^s_{\widetilde{C}_i, \widetilde{C}_{i+1}} \in
\mathfrak{q}_{\mathscr{R}^t}$.  Since $\mathfrak{q}_{\mathscr{R}^t}$ is prime, at least one factor lies
in $\mathfrak{q}_{\mathscr{R}^t}$.  The rule governing the transition
$\mathscr{R}^t \to \mathscr{R}^{t+1}$ selects such an index $i_0$: the
bottom entry of $\widetilde{C}_{i_0+1}$ in $R_s$ is recoloured red and
shifted into $R_{s+1}$, enforcing the vanishing of
$I^s_{\widetilde{C}_{i_0}, \widetilde{C}_{i_0+1}}$ on
$\mathfrak{u}(\mathscr{R}^{t+1})$.\\

Consequently, on $\mathfrak u(\mathscr R_t)$ one has the factorisation
\[
I^s_{C,C'}\big|_{\mathfrak u(\mathscr R_t)}
\;=\;
\prod_{i=0}^{n-1} I^s_{\widetilde C_i,\widetilde C_{i+1}}.
\]
Therefore %by the equation \ref{deg}
 one has $$\deg (I^s_{C,C'}\big|_{\mathfrak u(\mathscr R_t)})=\sum_{i=0}^{n-1} \deg(I^s_{\widetilde C_i,\widetilde C_{i+1}}).$$ 
\medskip
Now let $\mathscr C_t\subset \overline{B\cdot \mathfrak u(\mathscr R_t)}$ be an irreducible variety
on which $I^s_{C,C'}$ vanishes. Let $\mathfrak q_{\mathscr R_t}$ be the prime ideal defining $\mathscr C_t$. The above product decomposition shows that
\[
\prod_{i=0}^{n-1} I^s_{\widetilde C_i,\widetilde C_{i+1}}\in \mathfrak q_{\mathscr R_t},
\]
hence, since $\mathfrak q_{\mathscr R_t}$ is prime, at least one factor
$I^s_{\widetilde C_i,\widetilde C_{i+1}}$ belongs to $\mathfrak q_{\mathscr R_t}$.
The rule producing the red entry at stage $t$ selects this index $i$; equivalently, the passage
from $\mathscr R_t$ to $\mathscr R_{t+1}$ enforces the vanishing of the selected factor.

At the level of tableaux, $\mathscr R_{t+1}$ is obtained by taking the bottom entry of
$\widetilde C_{i+1}$, recolouring it red, and shifting it down by one row into $R_{s+1}$.
With this modification, both $I^s_{C,C'}$ and $I^s_{\widetilde C_i,\widetilde C_{i+1}}$ vanish on
$\mathfrak u(\mathscr R_{t+1})$ and thus on $\overline{B\cdot \mathfrak u(\mathscr R_{t+1})}$.

Iterating over $t=1,\dots,g$, the complete reverse tableau $\mathscr R_g$ yields a coherent choice of
a vanishing factor for each generator $I^{s_t}_{C_t,C_t'}$ of $\mathscr I_+$. Therefore the variety
defined by these imposed vanishing conditions is contained in $\mathscr N$, and by construction it
is $B$-stable and irreducible; hence it defines an element of $\Irr(\mathscr N)$.

\begin{theorem}%[{\cite[Thm.\,5.2]{FJ6}}]
               \label{thm:RT-records-factors} %[Reverse tableaux record vanishing factors]\label{thm:RT-records-factors}
At each stage $t$, the construction of $\mathscr R_t$ produces a distinguished red entry. For the
free pair implemented at that stage, this red entry singles out one of the $n$ factors in
Proposition~\ref{prop:free-pair-factorisation} and imposes its vanishing on the component under
construction. Consequently, a complete reverse tableau $\mathscr R=\mathscr R_g$ encodes a consistent
sequence of factor-vanishing conditions for the generators of $\mathscr I_+$, and hence determines
an element of $\Irr(\mathscr N)$.
\end{theorem}
\begin{proof}
Fix a stage $t$ and let $(C,C')$ be the free pair of height $s$ considered at this step. By
Proposition~\ref{prop:free-pair-factorisation}, on $\mathfrak u(\mathscr R_t)$ one has a product
decomposition of $I^s_{C,C'}$ into $n$ factors indexed by the height-$s$ columns in
 $\ell\mathcal T_{C,C'}(\mathscr R_t)$.

Let $\mathscr C_t\subset \overline{B\cdot \mathfrak u(\mathscr R_t)}$ be an irreducible variety on which
$I^s_{C,C'}$ vanishes. Writing $\mathfrak q_{\mathscr R_t}$ for the prime ideal defining $\mathscr C_t $, the product
decomposition shows that this product lies in $\mathfrak q_{\mathscr R_t}$, hence at least one factor lies in
$\mathfrak q_{\mathscr R_t}$. The algorithmic rule that produces the red entry at stage $t$ specifies which of
these factors is chosen to vanish; equivalently, passing from $\mathscr R_t$ to $\mathscr R_{t+1}$
adds precisely the relation forcing that factor to be zero on the next stage.

Iterating over $t=1,\dots,g$, the complete reverse tableau $\mathscr R_g$ yields a coherent choice of
a vanishing factor for each generator $I^{s_t}_{C_t,C_t'}$ of $\mathscr I_+$. Therefore the variety
defined by these imposed vanishing conditions is contained in $\mathscr N$, and by construction it
is $B$-stable and irreducible; hence it defines an element of $\Irr(\mathscr N)$.
%\end{proof}
%
%
%
%Reverse tableaux record vanishing factors.
%\begin{theorem}[
%                {\cite[Thm.\,5.2]{FJ6}}]\label{thm:RTrecords}
%At each stage $t$, the construction of $\mathscr{R}^{t+1}$ from $\mathscr{R}^t$
%singles out exactly one factor from the product in
%Proposition~\ref{prop:factor} and forces its vanishing on
%$\mathfrak{u}(\mathscr{R}^{t+1})$.  A complete reverse tableau
%$\mathscr{R} = \mathscr{R}^{\mathbf{g}}$ encodes a coherent sequence of such
%vanishing conditions  one for each generator of $\mathscr{I}_+$  and
%thereby determines an element $\mathscr{C}_{\mathscr{R}} \in \Irr(\mathscr{N})$.
%\end{theorem}
%
%\begin{proof}
%The preceding paragraph established the selection mechanism at a single stage.
%Iterating over $t = 1, \dots, \mathbf{g}$, the complete reverse tableau
%records a sequence of choices, one for each generator $I_{t}$ of $\mathscr{I}_+$
%in order.  At each step the resulting variety remains $B$-stable (by
%construction) and irreducible (since we are enforcing the vanishing of one
%factor of an already-zero product, picking from a prime ideal).  By
Theorem~\ref{thm:vanish}, all generators of $\mathscr{I}_+$ vanish on
$\mathfrak{u}(\mathscr{R}^{\mathbf{g}})$, so the final variety is contained
in $\mathscr{N}$.  Its irreducibility is established in
Proposition~\ref{prop_irreducibility} below via covering.
\end{proof}
\medskip

The following two propositions, Proposition~\ref{prop_dim} and
Proposition~\ref{prop_irreducibility}, were proved in \cite{FJ6}; their proofs are part of the author’s contribution to that work. They  show that every complete
reverse tableau gives rise to an irreducible component of the nilfibre. The
first establishes the containment in the nilfibre together with the expected
dimension, and the second proves irreducibility.
\begin{proposition}[]\label{prop_dim}
For every complete reverse tableau $\mathscr R$, one has
$\mathscr C_{\mathscr R}\subset \mathscr N$.
Moreover, $\dim \mathscr C_{\mathscr R}=\dim\mathfrak m-\mathbf g$.
\end{proposition}

\begin{proposition}[]\label{prop_irreducibility}
For every complete reverse tableau $\mathscr R$, the variety
$\mathscr C_{\mathscr R}$ is irreducible.
In particular, $\mathscr C_{\mathscr R}\in\Irr(\mathscr N)$.
\end{proposition}

%==========================================================================
\section{Codimension and Krull's Theorem}\label{sec:Krull}
%==========================================================================
The surjectivity argument requires precise control of codimensions.
Krull's principal ideal theorem provides the general codimension bound, and the
reverse-tableau construction shows that this bound is attained. This dimension equality is
what allows the constructed varieties to be recognised as irreducible components of the
nilfibre.
\subsection{Krull's principal ideal theorem}

We recall the algebraic tool underlying the codimension count; see
\cite{AM,Eis,Mat} for proofs.

Let $k$ be a field (algebraic closedness is not needed for the argument) \cite{AM}, \cite{Eis}, and let
$X$ be an irreducible affine $k$-variety of dimension $n$.
Set $A:=k[X]$. Then $A$ is a Noetherian integral domain and
\[
\dim A=\dim X=n.
\]

\subsection{Algebraic translation}
Let $D_1,\dots,D_g\subset X$ be hypersurfaces. Choose nonzero functions $f_i\in A$ such that
\[
D_i = V(f_i)\subset X,
\qquad
I:=(f_1,\dots,f_g)\subset A.
\]
Then
\[
D_1\cap\cdots\cap D_g \;=\; V(I),
\qquad
k[V(I)] \cong A/I,
\qquad
\dim\bigl(D_1\cap\cdots\cap D_g\bigr)=\dim(A/I).
\]
Since $\codim_X(Y)=\dim X-\dim Y$ for any irreducible closed subset $Y\subset X$, the desired inequality
\[
\codim_X\bigl(D_1\cap\cdots\cap D_g\bigr)\le g
\]
is equivalent to
\[
\dim(A/I)\ge n-g.
\]

\subsection{Krull's Principal Ideal Theorem}
\begin{theorem}[Krull, {\cite[Thm.\,11.14]{AM}}]\label{thm:Krull}
Let $R$ be a Noetherian ring and let $f\in R$.
If $\mathfrak q$ is a prime ideal minimal over $(f)$, then $ht(\mathfrak q)\le 1$.
Equivalently, every irreducible component of $V(f)\subset Spec(R)$ has codimension at most $1$.

%
%Let $A$ be a Noetherian ring with $\dim R=n$, and 
%let
%\[
%I=(f_1,\dots,f_g)\subset A .
%\]
%If $\mathfrak q$ is a prime ideal minimal over $I$, then
%\[
%\mathrm{ht}(\mathfrak q)\le g .
%\]
%Equivalently,
%\[
%\dim(A/I)\ge n-g .
%\]
%
%Geometrically, if $X=\mathrm{Spec}(A)$ and $D_i=V(f_i)$, then
%\[
%\codim_X(D_1\cap\cdots\cap D_g)\le g .
%\]
\end{theorem}

\medskip

Define $I_j:=(f_1,\dots,f_j)$ and $A_j:=A/I_j$ for $j=0,1,\dots,g$ (so $A_0=A$ and $A_g=A/I$).

{Base case ($g=1$).}
Applying Krull's theorem to $R=A$ and $f=f_1$ gives
\[
\dim(A_1)=\dim\bigl(A/(f_1)\bigr)\ge \dim(A)-1 = n-1.
\]

\paragraph{Inductive step.}
Assume $\dim(A_{g-1})\ge n-(g-1)$.
In the ring $A_{g-1}$, let $\overline f_g$ be the image of $f_g$.
Applying Krull's theorem to $R=A_{g-1}$ and $f=\overline f_g$ yields
\[
\dim(A_g)
=
\dim\bigl(A_{g-1}/(\overline f_g)\bigr)
\ge
\dim(A_{g-1})-1
\ge
\bigl(n-(g-1)\bigr)-1
=
n-g.
\]
Hence $\dim(A/I)=\dim(A_g)\ge n-g$, and therefore
$
%\boxed
{\;\codim_X\bigl(D_1\cap\cdots\cap D_g\bigr)\le g.\;}
$

\subsection{Equality case}
Equality $\codim_X(D_1\cap\cdots\cap D_g)=g$ holds precisely when the dimension drops by
exactly $1$ at each step, i.e.
\[
\dim(A_j)=\dim(A_{j-1})-1 \quad \text{for all } j=1,\dots,g.
\]
A standard sufficient (and in many settings equivalent) condition is that
$f_1,\dots,f_g$ form a \emph{regular sequence} in $A$, meaning that for each $j$ the class
$\overline f_j$ is not a zero divisor in $A_{j-1}=A/(f_1,\dots,f_{j-1})$.
Geometrically, this says that $D_j$ does not contain any irreducible component of
$D_1\cap\cdots\cap D_{j-1}$, so each hypersurface cuts the previous intersection properly.
\\
\medskip

In our setting, we are precisely in this situation, the invariant ideal $\mathscr I_{+}$ is
generated by the family of invariants $\{\,I_{C_i,C'_i}\,\}_i$, and where 
\[
\mathscr N := V(\mathscr I_{+}) \;=\; V\bigl(I_{C_1,C'_1},\dots,I_{C_g,C'_g}\bigr)\subset X.
\]
Moreover, for each $j=1,\dots,g$, the restriction of $I_{C_j,C'_j}$ to the previous
intersection
\[
X_{j-1}:=D_1\cap\cdots\cap D_{j-1}
\]
is not identically zero, as seen in section \ref{ RT} i.e.
\[
I_{C_j,C'_j}\big|_{X_{j-1}}\neq 0.
\]
Equivalently, $D_j=V(I_{C_j,C'_j})$ does not contain any irreducible component of $X_{j-1}$.
Consequently, each equation cuts the dimension by exactly one at step $j$, and we obtain
\[
\codim_X(\mathscr N)=g,
\]
namely the codimension of $\mathscr N$ equals the number of generating invariants.
We have seen that $\mathscr N$ need not be irreducible.
\begin{lemma}\label{lem:push-b-kills-invariant}
Let $(C,C')$ be a pair of \emph{neighbouring} columns of height $s$, and assume that $(C,C')$ is
free at stage $t$, i.e.\ $I^s_{C,C'}\big|_{\mathfrak u(\mathscr R_t)}\neq 0$.
Let $b$ be a black box (or a black subcolumn) contained in the interval $]C,C']$ and lying in a row
$R_i$ with $i<s$. Form a new tableau by moving $b$ to the left and pushing it down into some row
$R_j$ with $j>s$, and let $\mathfrak u_b\subset \mathfrak m$ be the subspace associated with this
new tableau (via the moving procedure).
Then
\[
I^s_{C,C'}\big|_{\mathfrak u_b}=0,
\qquad\text{and consequently}\qquad
\codim_{\mathfrak m}\,\overline{B\cdot \mathfrak u_b}\ge 1.
\]
\end{lemma}

\begin{proof}
By Proposition~\ref{prop:free-pair-factorisation}, the restriction of \(I^s_{C,C'}\) is given by the highest-degree minor of a submatrix \(X[I,J]\). In particular, each black entry in \(]C,C']\cap R^s\) contributes a matrix entry to \(X[I,J]\), whose row and column are indexed by that black entry.

Now let \(b\) lie in row \(R_i\) with \(i<s\). Move \(b\) down to a row \(R_j\) with \(j>s\), while shifting it to the left, and let \(\mathscr T_b\) be the resulting tableau and \(\mathfrak u_b\) the corresponding subspace. Upon restricting \(X[I,J]\) to \(\mathfrak u_b\)%, the unique entry contributed by \(b\) in the relevant row--column position disappears
. More precisely, if \(S\) denotes the set of entries displaced by \(b\) in \(\mathscr T_b\), then for every \(c\in S\) the corresponding coefficient vanishes:
$
x_{kc}=0 \qquad \text{whenever } k<c
$
and the entry \(k\) lies in \(\mathscr T_b\) either above \(c\) or to its right. It follows that theses new zero entries in the submatrix \(X[I,J]\) and its  highest-degree minor vanishes. Hence
$
I^s_{C,C'}\big|_{\mathfrak u_b}=0.
$ and therefore also on \(\overline{B\cdot \mathfrak u_b}\). Hence
\[
\overline{B\cdot \mathfrak u_b}\subset \mathcal Z(I^s_{C,C'})\subset \mathfrak m,
\]
where \(\mathcal Z(I^s_{C,C'})\) denotes the zero locus of \(I^s_{C,C'}\).
\end{proof}

The zero set of a semi-invariant is closely related to orbital varieties. An orbital varieties admit a tableau description: they are obtained by lowering a box (or a subcolumn) within the tableau. Since each orbital variety is represented by such a tableau, its codimension is determined by Corollary~\ref{codim}.

%\begin{lemma}
%An orbital variaty \mathcal O$ of codim s \le g contain a component of the nilfibre  \mathscr N, if and only if s invariant vanishes on  \mathcal O
%\end{lemma}
\medskip 

The following lemma \ref{lem:OV} connects the abstract minimal-prime problem with orbital varieties.  It tells us when an orbital-variety closure contains a component of the nilfibre, and therefore allows the surjectivity proof to pass between algebraic ideals and tableau moves.
\medskip 
To simplify notation in the following lemma, we write $I_i$ for $I_{C_i,C'_i}$.
\begin{lemma}\label{lem:OV}
Let $X=\overline{\mathcal O}\subset \mathfrak m$ be the closure of an orbital variety, 
$X$ is irreducible with $\codim_{\mathfrak m}(X)=s\le g$. Let $\mathscr N\subset \mathfrak m$ be the
$
\mathscr I_+=<I_1,\dots,I_g>_+\subset \mathbb C[\mathfrak m],
$
with $\codim_{\mathfrak m}(\mathscr N)=g$.
Set $A:=\mathbb C[X]=\mathbb C[\mathfrak m]/I(X)$ and write $\overline I_i$ for the image of $I_i$ in $A$.

Then $X$ contains an irreducible component of $\mathscr N$ if and only if the ideal
$(\overline I_1,\dots,\overline I_g)\subset A$ has a minimal prime of height $g-s$.
Equivalently, $X\cap \mathscr N$ has an irreducible component of codimension $g-s$ inside $X$.

In particular, this holds whenever there exist indices $i_1,\dots,i_{g-s}$ such that
$\overline I_{i_1},\dots,\overline I_{i_{g-s}}$ cut $X$ in codimension $g-s$ %(for instance, they forma regular sequence at the generic point of that component)
.
\end{lemma}

\begin{proof}
Let $\mathfrak q:=I(X)\subset \mathbb C[\mathfrak m]$, so $A=\mathbb C[\mathfrak m]/\mathfrak q$ and $ht(\mathfrak q)=s$.
The scheme-theoretic intersection $X\cap \mathscr N$ is defined in $A$ by the ideal
$(\overline I_1,\dots,\overline I_g)$. \cite{Har}

\smallskip
\emph{($\Rightarrow$)} Suppose $X$ contains an irreducible component $\mathscr C$ of $\mathscr N$.
Let $\mathfrak  p:=I(\mathscr C)$, so $\mathfrak  p$ is a minimal prime over $\mathscr I_+$ and $\mathscr C=V(\mathfrak p)$.
The inclusion $\mathscr C\subset X$ is equivalent to $\mathfrak q\subset \mathfrak p$. Passing to $A$, the prime
$\mathfrak p/\mathfrak q\subset A$ contains $(\overline I_1,\dots,\overline I_g)$ and is minimal over it.
Moreover, since \(R=\mathbb C[\mathfrak m]\) is a polynomial ring (catenary), we have
\[
ht(\mathfrak p/\mathfrak q)=ht(\mathfrak p)-ht(\mathfrak q).
\]
\[
ht(\mathfrak p/\mathfrak q)=ht(\mathfrak p)-ht(\mathfrak q)=g-s,
\]
since $ht(\mathfrak  p)=\codim_{\mathfrak m}(\mathscr C)=\codim_{\mathfrak m}(\mathscr N)=g$ and $ht(\mathfrak q)=s$.
Thus $(\overline I_1,\dots,\overline I_g)$ has a minimal prime of height $g-s$.

\smallskip
\emph{($\Leftarrow$)} Conversely, assume $(\overline I_1,\dots,\overline I_g)\subset A$ has a minimal
prime $\bar{ \mathfrak p}$ of height $g-s$. Let $\mathfrak p\subset \mathbb C[\mathfrak m]$ be its preimage. Then
$\mathfrak q\subset \mathfrak p$ and $\mathscr I_+\subset \mathfrak p$. Moreover,
\[
ht(\mathfrak p)=ht(\mathfrak q)+ht(\bar {\mathfrak p})=s+(g-s)=g.
\]
Since $g$ is the minimal possible height of a prime containing $\mathscr I_+$ (because
$\codim_{\mathfrak m}\mathscr N=g$), it follows that $\mathfrak  p$ is a minimal prime over $\mathscr I_+$.
Hence $V(\mathfrak p)$ is an irreducible component of $\mathscr N$, and it is contained in $V(\mathfrak  q)=X$.
\end{proof}

\begin{lemma}\label{lem:partial-RT-orbital}
Let $\mathscr R_t$ be a partial reverse tableau obtained after $t$ implementations.
Then $\mathscr R_t$ is associated with an orbital variety (equivalently,
$\overline{B\cdot \mathfrak u(\mathscr R_t)}$ is an orbital variety) if and only if, at every step
$r\le t$, no move sends a black entry one row downward and leftward % and to the left into a box lying strictly
below a red entry of $\mathscr R_t$.
\end{lemma}
\section{Surjectivity}\label{surjectivity}
We now turn to the proof of surjectivity. The component map defined by complete reverse
tableaux is surjective.
\subsection{The Global Red Multiset}\label{grm}

Let $C_{r}$ be a column of height $h$, and let $j_s\in R_s\cap C_{r}$.
We say that $j_s$ belongs to the {global red multiset} if the column $C_{r}$ is
surrounded $h-s+1$ neighbouring column pairs $(C_i,C_i')$ whose heights $s_i$ satisfy
$
s_i\in [\,s,h\,].
$

Moreover,   the entry $j_h\in R_h\cap C_{r}$ is said to have {maximal multiplicity $m'$}
if $C_{r}$ is enclosed by exactly $m'+1$ neighbouring column pairs $(C_i,C_i')$ with heights
$
s_i\in [\,h,\,h+m'\,].
$

The orbital varieties in Lemma~\ref{lem:OV} admit a tableau description. By
Lemma~\ref{lem:partial-RT-orbital}, this tableau may be viewed as a (partial) reverse tableau, in
which the entries that are lowered belong to the global red multiset.

\subsection{Reconstruction and surjectivity.}
%Finally one shows that every such factor-choice sequence coming from a minimal prime is realised by a reverse tableau, and that the resulting ideal recovers the prime.

\begin{theorem}\label{thm:surjectivity}
For every irreducible component $\mathscr C\in\Irr(\mathscr N)$ there exists a reverse tableau
$\mathscr  R$ such that $\phi(\mathscr  R)=\mathscr C_{\mathscr R}$.
\end{theorem}

\begin{proof}
Let $\mathscr C\in\Irr(\mathscr N)$, and write $\mathfrak p:=I(\mathscr C)$ for its defining
prime ideal, so $\mathfrak p$ is a minimal prime over $\mathscr I_+$.

\medskip

By Lemma~\ref{lem:OV}, there exists a finite family of orbital varieties
$\{\mathcal O_\alpha\}_{\alpha\in A}$ such that
\[
\mathscr C\subset \overline{\mathcal O_\alpha}\subset \mathfrak m\qquad(\alpha\in A),
\]
and we may choose such a family minimal for inclusion. Each $\overline{\mathcal O_\alpha}$ admits a
tableau description by lowering boxes (or subcolumns). By \S\ref{grm}, these lowered entries form a
multiset $Red_{\mathcal O_\alpha}$ contained in the global red multiset, and
\[
|Red_{\mathcal O_\alpha}|=\codim_{\mathfrak m}\,\overline{\mathcal O_\alpha},
\]
counted with multiplicities (Corollary~\ref{codim}). In this multiset $Red_{\mathcal O_\alpha}$  an entry may occur with multiplicity when the entree is lowered by more than one row relative to its original position in the Young tableau defining $\mathfrak m$. Moreover, the elements of $Red_{\mathcal O_\alpha}$  lie between neighbouring column pairs.

\medskip

Suppose that, for some \(\alpha \neq \beta\), the multisets \(Red_{\mathcal O_\alpha}\) and \(Red_{\mathcal O_\beta}\) intersect in a multiset of cardinality \(v>0\). \[
|Red_{\mathcal O_\alpha}\cap Red_{\mathcal O_\beta}|=v
\]
Then one may modify the young tableau of $\mathcal O_\beta$ (or $\mathcal O_\alpha$) by cancelling these $v$ common lowered
entries (i.e.\ move the corresponding lowered entries
back to their original rows), counted with multiplicity. This produces a new tableau and hence a new orbital variety
$\mathcal O_\beta'$ with
\[
\overline{\mathcal O_\beta}\subset \overline{\mathcal O_\beta'}\quad\text{and}\quad
|Red_{\mathcal O_\beta'}|=|Red_{\mathcal O_\beta}|-v,
\]
again by Corollary~\ref{codim}. Since $\overline{\mathcal O_\beta}\subset \overline{\mathcal O_\beta'}$,
we still have $\mathscr C\subset \overline{\mathcal O_\beta'}$. Replacing $\mathcal O_\beta$ by
$\mathcal O_\beta'$ and iterating this procedure, we obtain a minimal family
$\{\mathcal O_\alpha\}_{\alpha\in A}$ such that the multisets $Red_{\mathcal O_\alpha}$ are
pairwise disjoint:
\[
Red_{\mathcal O_\alpha}\cap Red_{\mathcal O_\beta}=\varnothing\qquad(\alpha\neq\beta).
\]

\medskip

Let $Red:=\bigsqcup_{\alpha\in A} Red_{\mathcal O_\alpha}$ be the disjoint multiset union. Then
\[
|Red|=\sum_{\alpha\in A}\codim_{\mathfrak m}\,\overline{\mathcal O_\alpha}.
\]
Because the moves are disjoint, Corollary~\ref{codim} and Lemma \ref{lem:OV} give additivity of codimension for the
corresponding intersections, and the minimality of the family implies that the intersection has
codimension $g$ (the codimension of $\mathscr N$) at the generic point of $\mathscr C$. Hence
\[
|Red|=g.
\]
In particular, $Red$ provides exactly $g$ red choices (with multiplicities) compatible with the
neighbouring-pair framework.

\medskip
Recall that $\phi(\mathscr R)=\mathscr C_{\mathscr R}$ by definition; thus it suffices to
construct a reverse tableau $\mathscr R$ such that $\mathscr C\subset \mathscr C_{\mathscr R}$.
Since $\mathscr C_{\mathscr R}$ is irreducible of codimension $g$ in $\mathfrak m$
(Propositions~\ref{prop_dim} and~\ref{prop_irreducibility}), and $\mathscr C$ has the same
dimension (because $\codim_{\mathfrak m}\mathscr N=g$), this inclusion will force
$\mathscr C=\mathscr C_{\mathscr R}$.

By \S\ref{grm} and the reverse-tableau algorithm (cf.\ \cite{FJ6}), the multiset $Red$ determines
a reverse tableau $\mathscr R$ (equivalently, a sequence of red entries $(a_1,\dots,a_g)$ in the
prescribed implementation order). Let $\mathscr C_{\mathscr R}$ be the irreducible subvariety
attached to $\mathscr R$ (Proposition~\ref{prop_irreducibility}), so $\mathscr C_{\mathscr R}\in\Irr(\mathscr N)$
and $\phi(\mathscr R)=\mathscr C_{\mathscr R}$ by definition.

By construction, at each stage the chosen red entry forces the vanishing of the corresponding
factor in the factorisation of the relevant generator (Theorem~\ref{thm:RT-records-factors}); hence
all generators of $\mathscr I_+$ vanish on $\mathscr C_{\mathscr R}$. Moreover, since the same red
data were extracted from orbital closures containing $\mathscr C$, we obtain an inclusion
\[
\mathscr C\subset \mathscr C_{\mathscr R}.
\]
Finally, Proposition~\ref{prop_dim} gives
\[
\dim \mathscr C_{\mathscr R}=\dim\mathfrak m-g=\dim \mathscr C,
\]
because $\mathscr C$ is an irreducible component of $\mathscr N$ of codimension $g$. Since
$\mathscr C$ and $\mathscr C_{\mathscr R}$ are irreducible and $\mathscr C\subset \mathscr C_{\mathscr R}$
with the same dimension, they coincide:
\[
\mathscr C=\mathscr C_{\mathscr R}.
\]
Thus $\phi$ is surjective.
\end{proof}

\bigskip

Let $\mathscr C\in\Irr(\mathscr N)$ and set $\mathfrak p:=I(\mathscr C)\subset \mathbb C[\mathfrak m]$.
Fix the complete sequence of neighboring column pairs $\{(C_t,C_t')\}_{t=1}^g$ and write
$I_t:=I^{s_t}_{C_t,C_t'}$.

We construct inductively a sequence of partial reverse tableaux
\[
\mathscr R_0,\mathscr R_1,\dots,\mathscr R_g
\]
such that, for each $t$, writing
\[
X_t:=\overline{B\cdot \mathfrak u(\mathscr R_t)}\subset \mathfrak m,
\]
one has
\begin{equation}\label{eq:surj-induction}
\mathscr C\subset X_t.
\end{equation}
For $t=0$, take $\mathscr R_0=\mathscr T$, hence $X_0=\mathfrak m$, and \eqref{eq:surj-induction} holds.

Assume $\mathscr R_t$ is constructed and \eqref{eq:surj-induction} holds.
Consider the next neighboring pair $(C_{t+1},C_{t+1}')$ of height $s$, and set $(C,C')=(C_{t+1},C_{t+1}')$.
Since $\mathscr C\subset \mathscr N=V(\mathscr I_+)$, we have $I^s_{C,C'}\in \mathfrak p$, hence $I^s_{C,C'}$
vanishes on $\mathscr C\subset X_t$.

Let $A_t:=\mathbb C[X_t]=\mathbb C[\mathfrak m]/I(X_t)$ %and denote images in $A_t$ by bars
.
Because $\mathscr C\subset X_t$, the ideal $\bar{\mathfrak p}:=\mathfrak p/I(X_t)\subset A_t$ is prime.

\smallskip
Let $(C,C')$ is free at stage $t$, then $I^s_{C,C'}\big|_{X_t}\not\equiv 0$. Choose $\mathscr R_{t+1}$ to be any reverse tableau obtained from
$\mathscr R_t$ by implementing $(C,C')$ %as in Section~\ref{2}
. By construction,
$X_{t+1}=\overline{B\cdot \mathfrak u(\mathscr R_{t+1})}\subset X_t$, so \eqref{eq:surj-induction} holds for $t+1$.

\smallskip
{ Let $(C,C')$ is free at stage $t$.}
By Proposition~\ref{prop:free-pair-factorisation}, the restriction of $I^s_{C,C'}$ to $X_t$ admits a factorisation
\[
\overline{I^s_{C,C'}}=\prod_{i=0}^{n-1}\overline{I^s_{\widetilde C_i,\widetilde C_{i+1}}}\qquad\text{in }A_t,
\]
where $\widetilde C_0,\dots,\widetilde C_n$ are the height-$s$ columns in $\mathcal T_{C,C'}(\mathscr R_t)$
(with $\widetilde C_0=C$ and $\widetilde C_n=C'$) and the factors are indexed by the
$\mathscr R_t$-pseudo-neighbouring pairs $(\widetilde C_i,\widetilde C_{i+1})$.

Since $\overline{I^s_{C,C'}}\in \bar{\mathfrak p}$ and $\bar{\mathfrak p}$ is prime, there exists an index $i$ such that
\[
\overline{I^s_{\widetilde C_i,\widetilde C_{i+1}}}\in \bar{\mathfrak p},
\]
equivalently $I^s_{\widetilde C_i,\widetilde C_{i+1}}\in \mathfrak p$.
Choose $\mathscr R_{t+1}$ among the reverse tableaux obtained from $\mathscr R_t$ by implementing $(C,C')$
so that the distinguished red entry selects this index $i$ (as in Theorem~\ref{thm:RT-records-factors}).
By construction% of the reverse-tableau step
, the selected factor vanishes on $\mathfrak u(\mathscr R_{t+1})$, hence on
$X_{t+1}=\overline{B\cdot \mathfrak u(\mathscr R_{t+1})}$. Since the same factor lies in $\mathfrak p$, it vanishes on
$\mathscr C$, and therefore $\mathscr C\subset X_{t+1}$. This proves \eqref{eq:surj-induction} for $t+1$.

\smallskip
By induction, \eqref{eq:surj-induction} holds for all $t$, and in particular for $t=g$ we obtain a reverse tableau
$\mathscr R:=\mathscr R_g$ such that
\[
\mathscr C\subset X_g=\overline{B\cdot \mathfrak u(\mathscr R)}=:\mathscr C_{\mathscr R}.
\]
By Propositions~\ref{prop_dim} and~\ref{prop_irreducibility}, one has $\mathscr C_{\mathscr R}\in \Irr(\mathscr N)$.
Since $\mathscr C$ is an irreducible component of $\mathscr N$ and
$\mathscr C\subset \mathscr C_{\mathscr R}\subset \mathscr N$, it follows that $\mathscr C=\mathscr C_{\mathscr R}$.

\medskip
The theorem \ref{thm:surjectivity} closes the gap between combinatorics and geometry.  Injectivity says that different component tableaux do not produce the same component; surjectivity says that no component is missing.  Hence the Red Set, through reverse tableaux and the component map, gives a complete and usable classification of the irreducible components of $\mathscr N$ in type $A$.

%==========================================================================
\section{Example}\label{sec:example}
%==========================================================================
We illustrate the surjectivity proof through three
explicit example, presented without graphical illustrations.  Throughout,
entries are written row-by-row (top to bottom) within each column.\\

\bigskip
%\subsection
{Example : The composition $(1,2,2,1)$}

Let the composition be $(1,2,2,1)$. Then $n=6$, and the standard tableau $\mathscr T$ is
\[
\begin{array}{cccc}
  C_1 & C_2 & C_3 & C_4\\
  1   & 2   & 4   & 6\\
      & 3   & 5   &
\end{array}
\]
where
$
C_1=(1), C_2=(2,3), C_3=(4,5),$ and $ C_4=(6).
$
There are two neighbouring pairs
$
(C_2,C_3) \text{of height }2,
\text{ and }
(C_1,C_4) \text{of height }1.
$
Hence $\mathbf g=2$.\\
The corresponding Benlolo--Sanderson invariants are
$$
I^2_{C_2,C_3}
=
x_{24}x_{35}-x_{25}x_{34},
$$
$$
I^1_{C_1,C_4}
=
x_{12}x_{24}x_{46}
+x_{12}x_{25}x_{56}
+x_{13}x_{34}x_{46}
+x_{13}x_{35}x_{56}.
$$
Where the degree of $I^2_{C_2,C_3}$ is $2$, and the degree of $I^1_{C_1,C_4}$ is $3$, in agreement with Lemma~\ref{lem:degree}.

\medskip
In this example the Global Red Set is
$
\operatorname{Global Red}(\mathscr T)=\{4,5,6\}.
$
Indeed, the entry $5$ belongs to the Global Red Set because it is the bottom entry of the column $C_3=(4,5)$, and $C_3$ is involved in the neighbouring pair $(C_2,C_3)$ of height $2$. The entry $6$ belongs to the Global Red Set because $C_4=(6)$ is involved in the neighbouring pair $(C_1,C_4)$ of height $1$. Finally, the entry $4$ belongs to the Global Red Set because it lies above $5$ in the column $C_3$, and the column $C_3$ is surrounded by the two relevant neighbouring pairs $(C_2,C_3)$ and $(C_1,C_4)$, of heights $2$ and $1$ respectively.

Thus a complete reverse tableau has a Red Set of cardinality $\mathbf g=2$, chosen from the Global Red Set:
\[
R(\mathscr R)\subset \operatorname{Global Red}(\mathscr T),
\qquad
|R(\mathscr R)|=2.
\]

\medskip
We have two  complete sequence :\\
The first order of complete sequence is $\mathcal P=\{(C_1,C_4),(C_2,C_3)\}$. 
We  start by implementing $(C_1,C_4)$ of height $1$. The lowest black entry of $C_4$ is $6$. We recolour $6$ red and place a new black copy of $6$ below the entry $4$, shifting the lower parts of the intermediate columns to the left to make space . This gives
\[
\mathscr R^1=\mathscr T
\quad\longmapsto\quad
\widetilde{\mathscr R}^2=
\begin{array}{cccc}
  C_1 & C_2 & C_3 & C_4\\
  1   & 2   & 4   & \red{6}\\
      3 &   5   & 6\\
      &     &   &
\end{array}.
\]
The excluded roots include
$
\{x_{13}, x_{25}, x_{46}\}.
$
Hence
$
I^1_{C_1,C_4}\big|_{\mathfrak u(\widetilde{\mathscr R}^2)}=0.
$

It remains to implement the neighbouring pair $(C_2,C_3)$ of height $2$.
At the stage $\widetilde{\mathscr R}^2$, the root $x_{25}$ is excluded, hence
$x_{25}\notin \mathfrak u(\widetilde{\mathscr R}^2)$, but it is in $\overline{B\cdot \mathfrak u(\widetilde{\mathscr R}^2)} $, the in variant $I^2_{C_2,C_3}$  restricted to  $\overline{B\cdot \mathfrak u(\widetilde{\mathscr R}^2)} $ isn't disturbe
\[
I^2_{C_2,C_3}|_{ \mathfrak u(\widetilde{\mathscr R}^2)}
=
x_{24}x_{35}-x_{25}x_{34}
\]

Implementing $(C_2,C_3)$ recolours the entry $5$ red  and put a black $5$ usder $3$ and produces the complete reverse tableau
\[
\widetilde{\mathscr R}^3
=
\begin{array}{cccc}
  C_1 & C_2 & C_3 & C_4\\
  1   & 2   & 4   & \red{6}\\
       3   & \red{5} &6 \\
       5   &    &
\end{array}.
\]
Its Red Set is
\[
R(\widetilde{\mathscr R}^3)=\{5,6\}.
\]
Thus the second order gives the same Red Set $\{5,6\}$ as one of the tableaux obtained from the first order. This illustrates the flexibility described in Remark~\ref{rmk:flexibility}: different complete sequences may lead to the same Red Set, and hence to the same irreducible component.
\\
\bigskip

The second order complete sequence is $\mathcal P'=\{(C_2,C_3),(C_1,C_4)\}$.
We first implement the neighbouring pair $(C_2,C_3)$ of height $2$. The lowest black entry of $C_3$ is $5$. We recolour $5$ red and place a new black copy of $5$ below 3 in column  $C_2$. Hence
\[
\mathscr R^1=\mathscr T
\quad\longmapsto\quad
\mathscr R^2=
\begin{array}{cccc}
  C_1 & C_2 & C_3 & C_4\\
  1   & 2   & 4   & 6\\
      & 3   & \red{5} & \\
      & 5   &     &
\end{array}.
\]
The excluded roots created at this step are 
${
x_{25}, x_{35}.
}$
Therefore
$
I^2_{C_2,C_3}\big|_{\mathfrak u(\mathscr R^2)}=0.
$
This is exactly the vanishing predicted by Theorem~\ref{thm:vanish}.

We now restrict the second invariant $I^1_{C_1,C_4}$ to $\mathfrak u(\mathscr R^2)$. Since $x_{25}=x_{35}=0$ on $\mathfrak u(\mathscr R^2)$, and   $\{x_{25},x_{35}\}\notin  \overline{B\cdot \mathfrak u(\widetilde{\mathscr R}^2)} $ ,  we obtain
\[
I^1_{C_1,C_4}\big|_{\mathfrak u(\mathscr R^2)}
=
x_{12}x_{24}x_{46}
+
x_{13}x_{34}x_{46}.
\]
Hence
\[
I^1_{C_1,C_4}\big|_{\mathfrak u(\mathscr R^2)}
=
\bigl(x_{12}x_{24}+x_{13}x_{34}\bigr)x_{46}.
\]
This is the factorisation predicted by Proposition~\ref{prop:factor}.

In the trapezium $\mathcal T^1_{C_1,C_4}(\mathscr R^2)$, the height-$1$ pseudo-neighbouring columns are
$$
\widetilde C_0=C_1=(1),
\widetilde C_1=(4), \text { and }
\widetilde C_2=C_4=(6).
$$
Thus the pseudo-neighbouring pairs are
$
(\widetilde C_0,\widetilde C_1),$ and $
(\widetilde C_1,\widetilde C_2).
$
The factorisation can therefore be written as
$$
I^1_{C_1,C_4}\big|_{\mathfrak u(\mathscr R^2)}
=
I^1_{\widetilde C_0,\widetilde C_1}
\,
I^1_{\widetilde C_1,\widetilde C_2},
\text{ where }
I^1_{\widetilde C_0,\widetilde C_1}
=
x_{12}x_{24}+x_{13}x_{34},
\qquad
I^1_{\widetilde C_1,\widetilde C_2}
=
x_{46}.
$$
Since the ideal of an irreducible component is prime, the vanishing of the product forces the vanishing of at least one of these two factors. The reverse tableau records which factor is selected, as in Theorem~\ref{thm:RT-records-factors}.

\medskip

\noindent
If the selected the entry is $6$ is recoloured red  and put the black $6$ under $4$ pushing  red$5$ to make space these lead to new excluded roo $\{x_{1,3},x_{1,5},x_{3,5},x_{2,5}, x_{4,6}\}$  . Therefor the factor is $I^1_{\widetilde C_1,\widetilde C_2}=x_{46}$ vanishes . We obtain
\[
\mathscr R^3_{(5,6)}
=
\begin{array}{cccc}
  C_1 & C_2 & C_3 & C_4\\
  1   & 2   & 4   & \red{6}\\
       3   & \red{5} &6 \\
       5   &   &
\end{array}.
\]
The Red Set is
$
R(\mathscr R^3_{(5,6)})=\{5,6\}.
$

\medskip

If the selected the entry $4$ is recoloured red  and put a black $4$ under $2$ pushing $(3,5)$ to make space.  The set of excluded roots are $\{x_{1,3},x_{1,5}, x_{2,5}, x_{2,4}\}$. Therefor the factor is 
$
I^1_{\widetilde C_0,\widetilde C_1}
=
x_{12}x_{24}+x_{13}x_{34},
$ vanishes. We obtain
\[
\mathscr R^3_{(5,4)}
=
\begin{array}{cccc}
  C_1 & C_2 & C_3 & C_4\\
  1   & 2   & \red{4} & 6\\
       3 & 4 & \red{5} & \\
       5   &   &
\end{array}.
\]
The Red Set is
$
R(\mathscr R^3_{(5,4)})=\{5,4\}.
$

Thus the first order $\mathcal P'=\{(C_2,C_3),(C_1,C_4)\}$ gives two possible Red Sets 
$
\{5,6\}
 \text{ and } 
\{5,4\}.
$

\medskip

%In conclusion, for the composition $(1,2,2,1)$, the Global Red Set is
%\[
%\operatorname{GlobalRed}(\mathscr T)=\{4,5,6\},
%\]
%while the possible Red Sets of complete reverse tableaux are
%\[
%\{5,6\}
%\qquad\text{and}\qquad
%\{5,4\}.
%\]
%These two Red Sets give the two irreducible components $\mathscr C$ and $\mathscr C'$ of the nilfibre $\mathscr N$, where  $\mathscr C$ is the zero of the ideal $<x_{24}x_{35}-x_{25}x_{34},x_{12}x_{24}+x_{13}x_{34}>$  and  $\mathscr C'$  is the zero of the ideal $<x_{24}x_{35}-x_{25}x_{34},x_{4,6}>$   . This example shows explicitly how the factorisation of the BS invariants, together with the primeness of the defining ideal of an irreducible component, is recorded by the reverse-tableau construction.

In conclusion, for the composition $(1,2,2,1)$, the Global Red Set is
\[
\operatorname{Global Red}(\mathscr T)=\{4,5,6\},
\]
whereas the possible Red Sets of complete reverse tableaux are
\[
\{5,6\}
\qquad\text{and}\qquad
\{5,4\}.
\]
These two Red Sets correspond to the two irreducible components of the
nilfibre $\mathscr N$.  The example shows, in a concrete case, how the
factorisation of the Benlolo--Sanderson invariants and the primeness of the
defining ideal of an irreducible component are recorded by the reverse-tableau
construction.

\section*{Impact and Concluding Remarks}

The results established in this paper complete the classification of the irreducible components of the nilfibre $\mathscr{N}$ in type $A$. By proving the surjectivity of the component map, we show that the combinatorial framework of reverse tableaux is sufficient to capture the entire geometric complexity of $\mathscr{N}$.

The impact of this work is significant for several reasons. Firstly, it provides a constructive method to reach every irreducible component, overcoming the challenges posed by the lack of finite orbits. Secondly, the \textbf{flexibility} inherent in mapping through the tableau and the Red Set  proves to be a vital visual and technical tool for the factorisation of invariants. This suggests that the "neighborhood philosophy" used here may serve as a blueprint for extending these results to other classical Lie algebras (types $B, C, D$). Finally, this work bridges a gap between the purely combinatorial study of tableaux and the algebraic geometry of orbital varieties, providing a more intuitive, visual language for future research in the field.

\medskip

\end{document}